\documentclass{article}
\usepackage{fontenc}
\usepackage{amsmath}
\usepackage{amssymb}
\usepackage{amsfonts}
\usepackage{amsthm}

\theoremstyle{plain}
\newtheorem{thm}{Theorem}[section]
\newtheorem{prop}[thm]{Proposition}
\newtheorem{cor}[thm]{Corollary}
\newtheorem{lem}[thm]{Lemma}

\theoremstyle{definition}
\newtheorem{defn}[thm]{Definition}

\begin{document}

\title{On Lattices over Valuation Rings of Arbitrary Rank}

\author{Shaul Zemel\thanks{This work was supported by the Minerva Fellowship
(Max-Planck-Gesellschaft).}}


\maketitle

\section*{Introduction}

A well-known result states that for any odd prime $p$, the isomorphism classes
of $p$-adic lattices correspond to the possible symbols of the form
$\prod_{e=0}^{N}(p^{e})^{\varepsilon_{e}n_{e}}$, where
$\varepsilon_{e}\in\{\pm1\}$ and $n_{e}\in\mathbb{N}$ for every $e$. Moreover,
the Witt Cancellation Theorem holds for $p$-adic lattices, as is shown in
\cite{[J1]}. The same assertions hold for lattices over the ring of integers
$\mathcal{O}$ in a finite field extension of $\mathbb{Q}_{p}$. In all
references known to the author (e.g., \cite{[J1]}, \cite{[J3]}, \cite{[C]},
\cite{[O]}, etc.) the proof is based on the $p$-adic valuation being
\emph{discrete}, or at least of rank 1 (see \cite{[D]}). Recall that a valuation
has rank 1 if the value group can be embedded in $(\mathbb{R},+)$ as an ordered
group. The first aim of this paper is to generalize these assertions to lattices
over any 2-Henselian valuation ring with a finite residue field whose
characteristic is not 2. Indeed, a very simple variation of the short argument
appearing in Section 4 of Chapter 1 of \cite{[MH]} suffices to prove this
result (see also Section 3 in Chapter 8 of \cite{[C]} for the case of lattices
over the ring $\mathbb{Z}_{p}$ of $p$-adic numbers).

We remark that several texts deal with non-unimodular lattices (also in the
Hermitian setting) under various degrees of generality (see, e.g., \cite{[BFF]}
or \cite{[BFFM]}). However, these references use abstract tools such as
quadratic forms over Hermitian categories. The book \cite{[Knu]} also deals with
related topics, but mostly in the quadratic setting, while we do not assume that
our bilinear form comes from a quadratic form (in the sense of \cite{[A]} in
characteristic 2---see more details in Section \ref{Inv}). The book \cite{[Kne]}
considers bilinear forms over valuation rings, but treats them up to isomorphism
over the quotient field rather than just over the valuation ring itself. It
seems that our results are independent of the results appearing in \cite{[Knu]}
and \cite{[Kne]}. In any case our proofs are very concrete and simple, and show
how the desired assertions can be obtained solely from 2-Henselianity.

Next we consider unimodular rank 2 lattices, which contain the only non-trivial
indecomposable lattices over valuation rings (up to multiplying the bilinear
form by a scalar). We focus on the residue characteristic 2 case, where indeed
such indecomposable lattices exist. We define an invariant for isomorphism
classes of these lattices, which in some sense generalizes the Arf invariant
defined in \cite{[A]}. We then show how two invariants characterize the
isomorphism classes of such lattices, in case they contain a primitive
element whose norm is divisible by 2. We conclude by giving some relations
between different Jordan decompositions (in residue characteristic 2) which
yield isomorphic lattices, taking a Jordan decomposition of a lattice to a
``more canonical'' one.

In Section \ref{Jordan} we prove the existence of Jordan decompositions over
any valuation ring, and show that an approximated isomorphism between lattices
over a 2-Henselian valuation ring is a twist of a true isomorphism. Section
\ref{Unique} proves the ``uniqueness of the symbol'' result. In Section
\ref{Uni2} we present the conventions for unimodular rank 2 lattices, with their
\emph{generalized Arf invariants}. Section \ref{Inv} considers isomorphisms
between unimodular rank 2 lattices containing a primitive vector of norm
divisible by 2, and shows that the \emph{fine Arf invariant} and the \emph{class
of minimal norms} characterize the isomorphism class of such lattices. Finally,
in Section \ref{Ch2Can} we define when one Jordan decomposition in residue
characteristic 2 is ``more canonical'' than another Jordan decomposition, and
present certain transformations of Jordan decompositions which makes them ``more
canonical''.

I am thankful to U. First for reading this manuscript and providing several
enlightening insights. Many thanks are due to the anonymous referee, whose
suggestions have helped to greatly improve many aspects of this paper, as well
as the clarity of several arguments.

\section{Jordan Decompositions \label{Jordan}}

Let $R$ be a commutative ring, and let $M$ be a finite rank free $R$-module with
a symmetric bilinear form. We denote the bilinear form by $(\cdot,\cdot):M
\times M \to R$, and for $x \in M$ we write $x^{2}$ for $(x,x)$ (this element
of $R$ is called the \emph{norm} of $x$). The bilinear form maps $M$ to the dual
module $M^{*}=Hom_{R}(M,R)$, and we call the bilinear form \emph{non-degenerate}
if this map $M \to M^{*}$ is injective. In this case we call $M$ an
\emph{$R$-lattice}. Note that our non-degeneracy condition is weaker than the
inner product condition considered in \cite{[MH]}, where the map $M \to M^{*}$
is required to be bijective. In case this map is bijective we call the lattice
$M$ \emph{unimodular}. We consider only the case where $R$ is an integral
domain, so we assume this from now on. In this case we can extend scalars to the
field of fractions $\mathbb{K}$ of $R$, and obtain a $\mathbb{K}$-lattice, or
equivalently an inner product space over $\mathbb{K}$. Then non-degeneracy is
equivalent to requiring a non-zero determinant for the Gram matrix of the
bilinear form using any basis for $M$.

Some authors (e.g., \cite{[MH]}) assume that a module underlying lattice is
projective (and not necessarily free). However, our main interest here is the
case where $R$ is a valuation ring, hence a local ring, where these two
conditions are equivalent.

Elements $x$ and $y$ of a lattice $M$ are called \emph{orthogonal} and denoted
$x \perp y$ if $(x,y)=0$. For a submodule $N$ of $M$ we denote $N^{\perp}$ its
\emph{orthogonal complement}, the submodule of $M$ consisting of those $x \in M$
such that $x \perp y$ for all $y\in N$. Our non-degeneracy condition is
equivalent to the assertion that $M^{\perp}=\{0\}$. A direct sum of lattices is
\emph{orthogonal} if every two elements from different lattices are orthogonal.
Then an orthogonal direct sum of bilinear form modules is a lattice (i.e.,
non-degenerate) if and only if all the summands are lattices. Two lattices $M$
and $N$ are called \emph{isomorphic}, denoted $M \cong N$, if there exists an
$R$-module isomorphism between them which preserves the bilinear form. A
non-degenerate lattice $M$ has an orthogonal basis if and only if it is
isomorphic to a direct sum of rank 1 lattices. For a lattice $M$ and an element
$0 \neq a \in R$, we denote $M(a)$ the lattice obtained from $M$ by multiplying
the bilinear form by $a$. This lattice is non-degenerate if and only if $M$ has
this property, and a basis for the module $M$ is orthogonal for the lattice $M$
if and only if it is orthogonal for $M(a)$.

We denote the group of invertible elements in the ring $R$ by $R^{*}$, and the
multiplicative group of $\mathbb{K}$ by $\mathbb{K}^{*}$. The determinant of a
Gram matrix of a basis of $M$ is independent of the choice of basis up to
elements of $(R^{*})^{2}$. Hence a statement of the form ``the determinant of
the bilinear form on $M$ divides an element of $R$'' is well-defined. We note
that $M$ is unimodular (hence non-degenerate in the sense of \cite{[MH]}) if
and only if its determinant is in $R^{*}$.

An element $x$ of a lattice $M$ is called \emph{primitive} if the module $M/Rx$
is torsion-free. This condition is equivalent to $x$ being an element of some
basis of $M$, and it is preserved under multiplication from $R^{*}$. Note that
this notion depends only on the structure of $M$ as an $R$-module, and not on
the bilinear form on $M$.

\smallskip

We call an $R$-lattice $M$ \emph{uni-valued} if it can be written as $L(\sigma)$
with $L$ unimodular and $\sigma \in R$. This notion (at least over valuation
rings) is closely related to the notion of $\mathfrak{a}$-unimodularity
considered, for example, in \cite{[O]}. A \emph{Jordan decomposition} of a
lattice $M$ is a presentation of $M$ as $\bigoplus_{k=1}^{t}M_{k}$ with $M_{k}$
uni-valued, such that if $M_{k}=L_{k}(\sigma_{k})$ with $L_{k}$ unimodular then
$v(\sigma_{k}) \neq v(\sigma_{l})$ wherever $k \neq l$. Note that if $M$ and $N$
are uni-valued with the same $\sigma$ then so is $M \oplus N$. Hence
the condition about the $\sigma_{k}$ having different valuations can be easily
achieved by successive combination of two uni-valued lattices with the same
$v(\sigma)$ into one.

We now prove the existence of Jordan decompositions for lattices over arbitrary
valuation rings. We follow closely the arguments in Chapter 1 of \cite{[MH]}
(where bilinear forms over fields are considered) and \cite{[J1]} or Chapter 8
of \cite{[C]} (which considers the $p$-adic numbers).

Let $N$ be a (free) submodule of $M$ which is non-degenerate of rank $r$. First
we prove a simplified version of Lemma 1 of \cite{[J1]}:

\begin{lem}
Let $e_{i}$, $1 \leq i \leq r$ be a basis for $N$, and let $A \in M_{r}(R)$ be
the matrix whose $ij$-entry is $(e_{i},e_{j})$. For any $x \in M$ and $1 \leq i
\leq r$, denote by $A_{i,x}$ the matrix whose $ij$-entry (with the same $i$) is
$(x,e_{j})$ and all the other entries coincide with those of $A$. If $\det A$
divides $\det A_{i,x}$ in $R$ for any $i$ and $x$ then $M$ decomposes as $N
\oplus N^{\perp}$ (as lattices). \label{sublat}
\end{lem}

\begin{proof}
Since $N$ is non-degenerate, we have $N \cap N^{\perp}=\{0\}$, hence $N \oplus
N^{\perp}$ is a sub-lattice of $M$. We need to show equality. Given $x \in M$,
we claim that there exists some $y=\sum_{i=1}^{r}a_{i}e_{i} \in N$ (with $a_{i}
\in R$) such that $(x,e_{j})=(y,e_{j})$ for any $1 \leq j \leq r$. Indeed, these
equalities (one for each $1 \leq j \leq r$) yield a system of linear equations
for the coefficients $a_{i}$, which we can solve over $\mathbb{K}$ since the
corresponding matrix is $A$ (hence of non-zero determinant). But the solution
is given using Cramer's formula, i.e., $a_{i}=\frac{\det A_{i,x}}{\det
A}\in\mathbb{K}$, and our assumptions imply that these coefficients are in $R$.
Now, since $y \in N$ and our assumption on $y$ implies $x-y \in N^{\perp}$, we
obtain that $x=y+(x-y) \in N \oplus N^{\perp}$, as desired. This proves the
lemma.
\end{proof}

Assume now that $R$ is a \emph{valuation ring}. This means that there is a
totally ordered (additive) group $\Gamma$ (the \emph{value group}) and a
surjective homomorphism $v:\mathbb{K}^{*}\to\Gamma$ (called the
\emph{valuation}) satisfying $v(x+y)\geq\min\{v(x),v(y)\}$ for every $x$ and $y$
in $\mathbb{K}$. Here and throughout, we extend $v$ to a function on
$\mathbb{K}$ by setting $v(0)=\infty$ and considering it larger than any element
of $\Gamma$. The statement that $R$ is the valuation ring of $v$ means that $R$
consists precisely of those elements $x\in\mathbb{K}$ such that $v(x)\geq0$
(with 0 here is the trivial element of $\Gamma$). For any $\gamma\in\Gamma$ we
define $I_{\gamma}=\big\{x\in\mathbb{K}^{*}\big|v(x)>\gamma\big\}$. It is a
(proper) ideal in $R$ if $\gamma\geq0$. In particular, $I_{0}$ is the unique
maximal ideal of $R$. We remark that an ideal of the sort
$\big\{x\in\mathbb{K}^{*}\big|v(x)\geq\gamma\big\}$ is just the principal
(perhaps fractional, if $\gamma<0$) ideal $\sigma R$ with $v(\sigma)=\gamma$,
hence requires no further notation.

In many references (e.g., \cite{[D]}), the ordered group $\Gamma$ is considered
as a subgroup of the additive group of $\mathbb{R}$. Such valuations are called
\emph{of rank 1}. In particular, the \emph{discrete} valuations, in which
$\Gamma\cong\mathbb{Z}$ (covering the case of the $p$-adic numbers and their
finite extensions) have rank 1. However, we pose no restrictions on $v$ or
$\Gamma$ in this paper, hence the rank is arbitrary.

For any $R$-lattice $M$ we define the \emph{valuation of $M$}, denoted $v(M)$,
to be $\min\{v(x,y)|x,y \in M\}$, where we use the shorthand $v(x,y)$ for
$v\big((x,y)\big)$. This value equals $\min_{i,j}v(e_{i},e_{j})$ wherever
$e_{i}$, $1 \leq i \leq r$ is a basis for $R$, since $(x,y)$ lies in the
$R$-module generated by these elements for any $x$ and $y$ in $R$. In
particular $v(M)$ is well-defined. If $M$ is uni-valued, then by writing
$M=L(\sigma)$ with $L$ unimodular we have $v(M)=v(\sigma)$. In a Jordan
decomposition $\bigoplus_{k=1}^{t}M_{k}$ of a lattice $M$ the condition on the
elements $\sigma_{k} \in R$ distinguishing the uni-valued components from being
unimodular reduces to the assertion that these components have different
valuations. We can thus assume, by changing the order if necessary, that
$v(M_{k})<v(M_{k+1})$ for every $1 \leq k<t$. Using these definitions, Lemma
\ref{sublat}
yields

\begin{prop}
Any lattice $M$ over a valuation ring $R$ admits a Jordan decomposition.
\label{decom}
\end{prop}

\begin{proof}
We apply induction on the rank of $M$. For rank 1 lattices the assertion is
trivial. Let $v=v(M)$. Assume first that there is an element $x \in M$ such
that  $v(x^{2})=v$. Then $N=Rx$ satisfies the condition of Lemma \ref{sublat},
so that we can write $M=N \oplus N^{\perp}$. On the other hand, if no such $x$
exists, then we take $x$ and $y$ in $M$ such that $v(x,y)=v$, and our assumption
implies $v(x^{2})>v$ and $v(y^{2})>v$. We claim that $x$ and $y$ are linearly
independent over $R$. Indeed, the equality $ax+by=0$ implies $ax^{2}+b(x,y)=0$
and $a(x,y)+by^{2}=0$, hence $a \in bI_{0}$ and $b \in aI_{0}$, which is
possible only if $a=b=0$. Moreover, $N=Rx \oplus Ry$ satisfies the condition of
Lemma \ref{sublat}. Indeed, the valuation of the
determinant is $2v$, while the valuation of any other $2\times2$ determinant
with entries in the image of the bilinear form has valuation at least $2v$. Thus
also here $M=N \oplus N^{\perp}$. It remains to verify that in both cases $N$ is
uni-valued. To see this, observe that any rank 1 lattice is uni-valued, and in
the second case dividing the bilinear form on $N$ by $(x,y)$ gives a unimodular
lattice. The induction hypothesis allows us to decompose $N^{\perp}$ into
uni-valued lattices, and adding $N$ to the component of valuation $v$ in
$N^{\perp}$ (if it exists) completes the proof of the proposition.
\end{proof}

The decomposition of Proposition \ref{decom} is called a \emph{Jordan splitting}
in \cite{[O]}. From the proof of Proposition \ref{decom} we deduce

\begin{cor}
If $2 \in R^{*}$ then the components $M_{k}$ have orthogonal bases. If $2 \notin
R^{*}$ then either $M_{k}$ has an orthogonal basis or it admits an orthogonal
decomposition into lattices of rank 2 each having a basis $\{x,y\}$ such that
$v(x^{2})$ and $v(y^{2})$ are both strictly larger than $v(x,y)$. \label{diag}
\end{cor}

\begin{proof}
First we show that if $2 \in R^{*}$ then there exists an element $x \in M$ with
$v(x^{2})=v(M)$. Indeed, if $v(x,y)=v$ and $v(2)=0$ while $v(x^{2})>v$ and
$v(y^{2})>v$ then $(x+y)^{2}=x^{2}+y^{2}+2(x,y)$ has valuation $v$. In view of
the proof of Proposition \ref{decom}, this proves the corollary in this case.
Assume now $2 \notin R^{*}$. The proof of Proposition \ref{decom} shows that
$M_{k}$ can be written as an orthogonal direct sum of rank 1 lattices and rank
2 lattices of the sort described above. It remains to show that if a lattice
of rank 1 appears in $M_{k}$ then $M_{k}$ has an orthogonal basis. It suffices
(by induction) to prove that if $N$ is the direct sum of one rank 1 lattice and
one rank 2 lattice of this form having the same valuation then $N$ has an
orthogonal basis. Let now $N=Rx \oplus Ry \oplus Rz$ be a lattice in which $x
\perp z$, $y \perp z$, and $(x,y)$ and $z^{2}$ have common (finite) valuation
$v$ while $v(x^{2})>v$ and $v(y^{2})>v$. One checks directly that the three
elements $tx+z$, $(z^{2})y-t(x,y)z$, and
$(y^{2}z^{2}+t^{2}(x,y)^{2})x-(x,y)(t^{2}x^{2}+z^{2})y-t(x^{2}y^{2}-(x,y)^{2})z$
form an orthogonal basis for $N$ for any $t \in R^{*}$. This proves the
corollary.
\end{proof}

\medskip

Recall that a valuation ring $R$ is called \emph{Henselian} if Hensel's Lemma
holds in $R$, namely if given three monic polynomials $f$, $g_{0}$, and $h_{0}$
in the polynomial ring $R[x]$ such that $f-g_{0}h_{0}$ lies in $I_{0}[X]$ (i.e.,
all the coefficients of that difference have positive valuation) and the
resultant of $g_{0}$ and $h_{0}$ is in $R^{*}$ then there exist monic
polynomials $g$ and $h$ in $R[x]$ such that $f=gh$ and $g-g_{0}$ and $h-h_{0}$
are in $I_{0}[X]$. In particular, taking $g$ to be of degree 1 renders this
statement equivalent to the assertion that if $a \in R$ and monic $f \in R[x]$
satisfy $v(f(a))>0$ and $v(f'(a))=0$ then $f$ has a root $b \in R$ with $b-a \in
I_{0}$. We call a valuation ring $R$ \emph{2-Henselian} if the last assertion
holds for any polynomial $f$ of degree 2, and if $2\neq0$ in $R$. We note that a
more general assertion holds in a Henselian ring, stating that for $f \in R[x]$
(not necessarily monic!) and an element $a \in R$ such that $v(f(a))>2v(f'(a))$
there exists a root $b$ of $f$ with $b-a \in I_{v(f'(a))}$. Indeed, following
the proof of the equivalence of $(e)$ and $(f)$ in Theorem 18.1.2 of \cite{[E]},
we use the Taylor expansion to write $f\big(a-\frac{f(a)}{f'(a)}y\big)$ as
$f(a)\big(1-y+\frac{f(a)}{f'(a)^{2}}y^{2}g(y)\big)$ for some polynomial $g$, and
we present this expression as $f(a)y^{d}h\big(\frac{1}{y}\big)$ where $d$ is the
degree of $f$ and $h(x)$ has the form $x^{d}-x^{d-1}+\sum_{i=0}^{d-2}c_{i}x^{i}$
with $c_{i} \in I_{0}$. Since $v(h(1))>0$ and $v(h'(1))=0$ there exists a root
$\lambda\in1+I_{0} \subseteq R^{*}$ of $h$, so that the required root of $f$ is
$b=a-\frac{f(a)}{f'(a)}\cdot\frac{1}{\lambda}$. Since $\lambda \in R^{*}$ we
know that $v(b-a)=v\big(\frac{f(a)}{f'(a)}\big)$. In a 2-Henselian valuation
ring this more  general condition holds for any $f \in R[x]$ of degree 2. As
Theorem 7 in Chapter 2 of \cite{[Sch]} shows that every complete valuation ring
is Henselian (and in particular 2-Henselian), our results hold for a variety of
interesting valuation rings.

\smallskip

The following property of 2-Henselian rings will be used below.

\begin{lem}
An element of $R$ of the form $1+y$ with $v(y)>2v(2)$ lies in $(R^{*})^{2}$, and
has a unique square root $1+z$ such that $v(z)>v(2)$. Moreover, the equality
$v(y)=v(z)+v(2)$ holds in this case. Let $A$, $B$, and $C$ be three elements of
$\mathbb{K}$ such that $v(AC)>2v(B)$ (hence $B\neq0$). Then the equation
$At^{2}+Bt+C=0$ has one solution in $\mathbb{K}$ with valuation
$v\big(\frac{C}{B}\big)$ (so that this solution is in $R$ if $v(C)
\geq v(B)$). If $A\neq0$ then the other solution has valuation
$v\big(\frac{B}{A}\big)$, which is strictly smaller. \label{quadeq}
\end{lem}

\begin{proof}
Consider the polynomial $f(t)=t^{2}-1-y$ and the approximate root 1. The
2-Henselianity of $R$ yields a root of this polynomial, which we write as $1+z$,
such that $z \in I_{v(2)}$ since $f'(1)=2$. We also have
$v(z)=v\big(\frac{f(1)}{f'(1)}\big)=v(y)-v(2)$. The second square root of $1+y$
is $-1-z$, and by subtracting 1 we obtain the element $-2-z$ of $R$, which has
valuation precisely $v(2)$ since $v(z)>v(2)$. Next, the equation $At^{2}+Bt+C=0$
has only one solution $t=-\frac{C}{B}$ if $A=0$,
and otherwise its two solutions are given by $\frac{-B\pm\sqrt{B^{2}-4AC}}{2A}$
(in an appropriate extension of $\mathbb{K}$ if necessary). Now,
$B^{2}-4AC=B^{2}(1+y)$ for $y=-\frac{4AC}{B^{2}}$, and the inequality
$v(y)>2v(2)$ allows us to write $\sqrt{B^{2}-4AC}$ as $B(1+z)$ with $z \in R$
such that $v(z)=v(y)-v(2)=v\big(\frac{2AC}{B^{2}}\big)$.  The two solutions
$\frac{B}{2A}\big(-1\pm(1+z)\big)$ of the equation lie in $\mathbb{K}$. The
solution with $+$ is $\frac{Bz}{2A}$ and has the required valuation
$v\big(\frac{C}{B}\big)$ (like in the case $A=0$), and the other solution has
valuation $v\big(\frac{B}{A}\big)$ since $v(2-z)=v(2)$. As $v(AC)>v(B^{2})$
implies $v\big(\frac{C}{B}\big)>v\big(\frac{B}{A}\big)$, this proves the lemma.
\end{proof}
Note that the condition $2\neq0$ was implicitly used in Lemma \ref{quadeq}, in
assumptions in which some elements have valuations strictly larger than $v(2)$,
as well as dividing by 2 in $\mathbb{K}$. All the assertions of Lemma
\ref{quadeq} collapse if $2=0$ in $R$.

\medskip

Our first result states that the existence of an approximate isomorphism between
lattices over 2-Henselian valuation rings implies that the lattices are indeed
isomorphic. See Theorem 2 of \cite{[D]} for the special case of complete
valuation rings of rank 1, and Corollary 36a of \cite{[J3]} or Lemma 5.1 of
\cite{[C]} for the case $R=\mathbb{Z}_{p}$.

\begin{thm}
Let $M$ and $N$ be $R$-lattices. Decompose $M$ as in Proposition \ref{decom},
and assume that $M$ and $N$ are isomorphic when we reduce modulo
$I_{v(M_{t})+2v(2)}$. Then $M \cong N$ as $R$-lattices. \label{approxiso}
\end{thm}

\begin{proof}
Denote $I_{v(M_{t})+2v(2)}$ by $I$. An isomorphism over $R/I$ can be lifted to
an $R$-module homomorphism $\varphi:M \to N$ (since $M$ is a free module).
Moreover, $\varphi$ must be bijective: Observe that $M$ and $N$ must have the
same rank, and by choosing bases for both modules the determinant of $\varphi$
is a unit modulo $I$ hence lies in $R^{*}$. $\varphi$ preserves the bilinear
form up to $I$, and we now show how to alter $\varphi$ to a lattice isomorphism
from $M$ to $N$. We apply induction on the (common) rank of $M$ and $N$.

Assume first that $M_{1}$ has an orthogonal basis, and let $x$ be an element of
the basis of $M_{1}$. Then $v(x^{2})$ is minimal in $M$, and the fact that
$x^{2} \notin I$ implies the equality $v(\varphi(x)^{2})=v(x^{2})$. This
valuation is also minimal in $N$. Moreover, the inequality
$v\big(\frac{x^{2}}{\varphi(x)^{2}}-1\big)>2v(2)$ holds, so that by Lemma
\ref{quadeq} there is $c \in R^{*}$ with $v(c-1)>v(2)$ such that
$c^{2}=\frac{x^{2}}{\varphi(x)^{2}}$. Lemma \ref{sublat} implies that
$M=Rx\oplus(Rx)^{\perp}$, and we define $\psi:M \to N$ by taking $x$ to
$c\varphi(x)$ and $u\in(Rx)^{\perp}$ to
$\varphi(u)-\frac{(\varphi(u),\varphi(x))}{\varphi(x)^{2}}\varphi(x)$. Since
$\big(\varphi(u),\varphi(x)\big) \in I$ (as $(u,x)=0$), we have
$\frac{(\varphi(u),\varphi(x))}{\varphi(x)^{2}} \in I_{2v(2)} \subseteq R$. Now,
$\psi(x)^{2}=x^{2}$, and if $u \perp x$ then we have $\psi(u)\perp\psi(x)$. If
$w \in M$ is another vector such that $w \perp x$ as well, then we have
\[\big(\psi(u),\psi(w)\big)=\big(\varphi(u),\varphi(w)\big)-\frac{(\varphi(u),
\varphi(x))(\varphi(w),\varphi(x))}{\varphi(x)^{2}}.\] The congruence
$\big(\varphi(u),\varphi(w)\big)\equiv(u,w)(\mathrm{mod\ }I)$ and the relations
$\frac{(\varphi(u),\varphi(x))}{\varphi(x)^{2}} \in R$ and
$\big(\varphi(w),\varphi(x)\big) \in I$ now imply
$\big(\psi(u),\psi(w)\big)\equiv(u,w)(\mathrm{mod\ }I)$ for any $u$ and $w$ in
$(Rx)^{\perp}$.

On the other hand, if $M_{1}$ has no orthogonal basis, then we take some $x$
and $y$ in $M_{1}$ such that $v(x,y)$ is minimal in $M$, and then
$v\big(\varphi(x),\varphi(y)\big)=v(x,y)$ is minimal in $N$. Moreover,
$v(\varphi(x)^{2})$ and $v(\varphi(y)^{2})$ are both larger than $v(x,y)$. We
need to modify $\varphi(x)$ and $\varphi(y)$ in order to obtain elements
spanning a rank 2 sublattice of $N$ which is isomorphic to $Rx \oplus Ry$. We
claim that there exist elements $s$ and $t$ in $I_{v(2)}$ and
$c\equiv1(\mathrm{mod\ }I_{v(2)})$ such that
\[\big(c\varphi(x)+cs\varphi(y)\big)^{2}=x^{2},\quad
\big(\varphi(y)+t\varphi(x)\big)^{2}=y^{2},\] and
\[\big(c\varphi(x)+cs\varphi(y),\varphi(y)+t\varphi(x)\big)=(x,y).\] First we
apply Lemma \ref{quadeq} with the numbers $A=\varphi(x)^{2}$,
$B=2\big(\varphi(x),\varphi(y)\big)$, and $C=\varphi(y)^{2}-y^{2} \in I$ (these
numbers satisfy the assumptions of that Lemma). The corresponding solution $t$,
of valuation $v(\frac{C}{B}\big)>v(2)$, satisfies
$\big(\varphi(y)+t\varphi(x)\big)^{2}=y^{2}$ as required.

Further, denote $x^{2}y^{2}-(x,y)^{2}$ by $\Delta$ and
$\varphi(x)^{2}\varphi(y)^{2}-\big(\varphi(x),\varphi(y)\big)^{2}$ by
$\Delta_{\varphi}$, so that $v(\Delta)=v(\Delta_{\varphi})=2v(x,y)$ (see
Corollary \ref{diag}). Observe that
\[\big(\varphi(x),\varphi(y)\big)^{2}-(x,y)^{2}=\Big(\big(\varphi(x),
\varphi(y)\big)-(x,y)\Big)\Big(\big(\varphi(x),\varphi(y)\big)+(x,y)\Big)\] and
\[\varphi(x)^{2}\varphi(y)^{2}-x^{2}y^{2}=\varphi(x)^{2}\big(\varphi(y)^{2}-y^{2
}\big)+y^{2}\big(\varphi(x)^{2}-x^{2}\big)\] are elements of $(x,y)I$, so that
$\Delta_{\varphi}-\Delta$ lies in the same ideal and
$x^{2}\big(\frac{\Delta_{\varphi}}{\Delta}-1\big) \in I$. Hence
$C=\varphi(x)^{2}-x^{2}\frac{\Delta_{\varphi}}{\Delta} \in I$, while
$B=2\big(\varphi(x),\varphi(y)\big)+2tx^{2}\frac{\Delta_{\varphi}}{\Delta}$ has
valuation $v(x,y)+v(2)$ and
$A=\varphi(y)^{2}-t^{2}x^{2}\frac{\Delta_{\varphi}}{\Delta}$ has valuation
$v(x^{2}) \geq v(x,y)$. Thus, we can use Lemma \ref{quadeq} again and obtain a
solution $s$, of valuation $v(\frac{C}{B}\big)>v(2)$, to $As^{2}+Bs+C=0$.
Furthermore, since $s$ and $t$ are in $I_{v(2)}$, the number
$\frac{(1+st)(\varphi(x),\varphi(y))+s\varphi(y)^{2}+t\varphi(x)^{2}}{(x,y)}$
is congruent to 1 modulo $I_{v(2)}$ (hence lies in $R^{*}$). We denote by $c$
the inverse of this number, hence $v(c-1)>v(2)$ as well. The two elements
$c\varphi(x)+cs\varphi(y)$ and $\varphi(y)+t\varphi(x)$ span $R\varphi(x) \oplus
R\varphi(y)$ since the determinant $c(1-st)$ of the transition matrix is in
$R^{*}$, and the choice of $c$ implies
$\big(c\varphi(x)+cs\varphi(y),\varphi(y)+t\varphi(x)\big)=(x,y)$.

In order to evaluate $\big(c\varphi(x)+cs\varphi(y)\big)^{2}$ we write the
square of the denominator of $c$ as
\[\Big[s^{2}\big(\varphi(x),\varphi(y)\big)^{2}+2s\varphi(x)^{2}\big(\varphi(x),
\varphi(y)\big)+\big(\varphi(x)^{2}\big)^{2}\Big]t^{2}+\]
\[+2\Big[s^{2}\varphi(y)^{2}\big(\varphi(x),\varphi(y)\big)+s\varphi(x)^{2}
\varphi(y)^{2}+s\big(\varphi(x),\varphi(y)\big)^{2}+\varphi(x)^{2}
\big(\varphi(x),\varphi(y)\big)\big]t+\]
\[+\Big[s^{2}\big(\varphi(y)^{2}\big)^{2}+2s\varphi(y)^{2}\big(\varphi(x),
\varphi(y)\big)+\big(\varphi(x),\varphi(y)\big)^{2}\Big].\] Substituting the
quadratic equation for $t$ in each of the coefficients of $s^{2}$, $s$, and 1
takes the latter expression to the form
\[\big[y^{2}\varphi(y)^{2}-\Delta_{\varphi}t^{2}\big]s^{2}+2\big[y^{2}
\big(\varphi(x),\varphi(y)\big)+\Delta_{\varphi}t\big]s+\big[y^{2}\varphi(x)^{2}
-\Delta_{\varphi}\big].\] Multiplying
$\big(c\varphi(x)+cs\varphi(y)\big)^{2}-x^{2}$ by the latter expression yields
(recall the numerator $(x,y)$ of $c$)
\[\big[(x,y)^{2}\varphi(y)^{2}-x^{2}y^{2}\varphi(y)^{2}+x^{2}\Delta_{\varphi}t^{
2}\big]s^{2}+\]
\[+2\big[(x,y)^{2}\big(\varphi(x),\varphi(y)\big)-x^{2}y^{2}\big(\varphi(x),
\varphi(y)\big)-x^{2}\Delta_{\varphi}t\big]s+\]
\[+\big[(x,y)^{2}\varphi(x)^{2}-x^{2}y^{2}\varphi(x)^{2}+x^{2}\Delta_{\varphi}
\big].\] The coefficients of $s^{2}$, $s$, and 1 are
$t^{2}x^{2}\Delta_{\varphi}-\varphi(y)^{2}\Delta$,
$-2tx^{2}\Delta_{\varphi}-2(\varphi(x),\varphi(y))\Delta$, and
$x^{2}\Delta_{\varphi}-\varphi(x)^{2}\Delta$ respectively, so the quadratic
equation for $s$ shows that the latter expression vanishes. This shows that
$\big(c\varphi(x)+cs\varphi(y)\big)^{2}=x^{2}$ as desired.

Let $u \in (Rx \oplus Ry)^{\perp}$ be given. As in the proof of Lemma
\ref{sublat}, we can find, using Cramer's rule, the coefficients of $\varphi(x)$
and $\varphi(y)$ which should be subtracted from $\varphi(u)$ in order to obtain
a vector perpendicular to $\varphi(x)$ and $\varphi(y)$. These coefficients are
of the form $\frac{\det A_{i,u}}{\Delta}$, hence lie in $R$, and in fact in
$I_{2v(2)}$. We define a map $\psi:M \to N$ by sending $x$ to
$c\varphi(x)+cs\varphi(y)$, $y$ to $\varphi(y)+t\varphi(x)$, and $u \in (Rx
\oplus Ry)^{\perp}$ to $\varphi(u)$ modified by the appropriate multiples of
$\varphi(x)$ and $\varphi(y)$. The map $\psi$ is an isomorphism of $Rx \oplus
Ry$ onto its image $R\varphi(x) \oplus R\varphi(y)$ and it takes $(Rx \oplus
Ry)^{\perp}$ onto the orthogonal complement of the latter space. In addition,
$\big(\psi(u),\psi(w)\big)\equiv(u,w)(\mathrm{mod\ }I)$ for every $u$ and $w$
in $(Rx \oplus Ry)^{\perp}$ by arguments similar to the previous case, using the
orthogonality of $\psi(u)$ and $\psi(w)$ to $\varphi(x)$ and to $\varphi(y)$.

In both cases $M$ decomposes as $K \oplus K^{\perp}$ and we have altered
$\varphi$ to a map $\psi$ which is an isomorphism on $K$ and preserves
the orthogonality between $K$ and $K^{\perp}$. Since the restriction of $\psi$
to $K^{\perp}$ (denoted $\psi\big|_{K^{\perp}}$) becomes an isomorphism when
reducing modulo $I$, the induction hypothesis allows us to alter
$\psi\big|_{K^{\perp}}$ to an isomorphism $\eta:K^{\perp}\to\psi(K^{\perp})$.
The map which takes $x \in K$ to $\psi(x)$ and $u \in K^{\perp}$ to $\eta(u)$ is
the desired isomorphism from $M$ to $N$.
\end{proof}

We note that in each induction step in Theorem \ref{approxiso} the element $c-1$
of $R$, as well as $s$ and $t$ in the second case, lie in $I_{v(2)}$. Moreover,
the coefficients we use when changing the map on the orthogonal complement in
each step lie in $I_{2v(2)}$. This proves the stronger statement, that reducing
any ``isomorphism-up-to-$I$'' modulo the (larger) ideal $I_{v(2)}$ yields the
image (modulo $I_{v(2)}$) of a true isomorphism from $M$ to $N$.

\section{Uniqueness of the Decomposition if $2 \in R^{*}$ \label{Unique}}

We recall that for odd $p$ there are two isomorphism classes of unimodular
$p$-adic lattices of rank $n$, and the isomorphism classes correspond to the
possible values of the Legendre symbol of the discriminant of the lattice over
the prime $p$ (see, for example, Section 3 of \cite{[Z]}). Then the
decomposition of a general $p$-adic lattice as described in Section \ref{Jordan}
allows us to define the \emph{symbol} of the $p$-adic lattice, which is an
expression of the form $\prod_{k=0}^{m}(p^{k})^{\varepsilon_{k}n_{k}}$ with
$n_{k}\in\mathbb{N}$ being the rank of the uni-valued component of valuation
$k$ and $\varepsilon_{k}\in\{\pm1\}$ is the appropriate Legendre symbol.
Moreover, the possible symbols are in one-to-one correspondence with isomorphism
classes of $p$-adic lattices. We now use a simple argument to show that a
similar assertion holds over any 2-Henselian valuation ring (of arbitrary rank)
with a finite residue field in which 2 is invertible.

\medskip

Following Section 4 of \cite{[MH]}, we define, for a decomposition of $M$ as an
orthogonal direct sum $K \oplus L$, the \emph{reflection} corresponding to this
decomposition to be the map $r:M \to M$ which takes $x \in K$ to $x$ and $y \in
L$ to $-y$. This map is an involution, which preserves the bilinear form on $M$.
First we prove

\begin{lem}
Let $R$ be a valuation ring in which $2 \in R^{*}$, let $M$ be an $R$-lattice,
and let $x$ and $y$ be elements in $M$ having the same norm. Let $v$ be the
valuation of this common norm, and assume further that the norm of any $z \in M$
is at least $v$. Then there is a reflection on $M$ taking $x$ to $y$.
\label{refl}
\end{lem}

\begin{proof}
Write $x=u+w$ and $y=u-w$ for $u=\frac{x+y}{2}$ and $w=\frac{x-y}{2}$. The norm
equality $x^{2}=y^{2}$ implies $u \perp w$, hence this common norm equals
$u^{2}+w^{2}$. Under our assumption on $v$ we have $v(u^{2}) \geq v$, $v(w^{2})
\geq v$, and $v(u^{2}+w^{2})=v$, whence at least one of the two inequalities is
an equality. Since $2 \in R^{*}$, the proof of Corollary \ref{diag} shows that
the 1-dimensional sublattice generated by the corresponding element ($u$ or $w$)
satisfies the conditions of Lemma \ref{decom}, giving a decomposition of $M$.
Observing again that $u \perp w$, we find that if $v(u^{2})=v$ then the
reflection with respect to the decomposition $M=Ru\oplus(Ru)^{\perp}$ gives the
desired outcome, while if $v(w^{2})=v$ we can use the one corresponding to the
decomposition $M=(Rw)^{\perp} \oplus Rw$. This proves the lemma.
\end{proof}

As an application of Lemma \ref{refl} we deduce

\begin{cor}
Every automorphism of a rank $n$ lattice $M$ over a valuation ring $R$ such that
$2 \in R^{*}$ is the composition of at most $n$ reflections.
\label{refgen}
\end{cor}

\begin{proof}
We apply induction on $n$, the case $n=1$ being trivial (since $Aut(M)$ is just
$\{\pm1\}$ in this case). Corollary \ref{diag} yields an orthogonal basis for
$M$, and let $x$ be an element of this basis whose norm has minimal valuation.
Given an automorphism $f$ of $M$, the elements $x$ and $f(x)$ of $M$ satisfy the
conditions of Lemma \ref{refl}, hence there exists a reflection $r$ on $M$
taking $f(x)$ to $x$. Lemma \ref{sublat} implies $M=Rx\oplus(Rx)^{\perp}$, and
the composition $r \circ f$ fixes $x$, hence restricts to an automorphism of
$(Rx)^{\perp}$. By the induction hypothesis, the latter automorphism is a
composition of at most $n-1$ reflections on $(Rx)^{\perp}$, and by extending
each such reflection to $M$ by leaving $x$ invariant we obtain that $r \circ f$
is the composition of at most $n-1$ reflections on $M$. Composing with
$r^{-1}=r$ completes the proof of the corollary.
\end{proof}

\smallskip

Next we prove a special case of the Witt Cancellation Theorem, which holds for
lattices over any valuation ring in which 2 is invertible.

\begin{prop}
Let $M$ and $N$ be lattices over a valuation ring $R$ in which $v(2)=0$, and
define $v=\min\{v(M),v(N)\}$. Let $L$ be a rank 1 lattice spanned by an element
$x$ whose norm has valuation not exceeding $v$, and assume that $M \oplus L$ and
$N \oplus L$ are isomorphic. Then the lattices $M$ and $N$ are isomorphic.
\label{cancr1}
\end{prop}

\begin{proof}
Let $f:M \oplus L \to N \oplus L$ be an isomorphism. The elements $0_{N}+x$ and
$f(0_{M}+x)$ of $N \oplus L$ satisfy the conditions of Lemma \ref{refl},
yielding a reflection $r$ on $N \oplus L$ taking the latter element to the
former. Writing $g=r \circ f$, we obtain an isomorphism from $M \oplus L$ to $N
\oplus L$ which takes the direct summand $L$ of the first lattice onto the
direct summand $L$ in the second one. This isomorphism must therefore take $M$
isomorphically onto $N$, which proves the proposition.
\end{proof}

Proposition \ref{cancr1} generalizes a special case of Theorem 1 of \cite{[J1]},
with a simpler proof.

\smallskip

We can now prove the main result for the case $v(2)=0$:

\begin{thm}
Let $M$ and $N$ be lattices over a 2-Henselian valuation ring $R$ such that $2
\in R^{*}$. Decompose $M$ and $N$, using Proposition \ref{decom}, as
$\bigoplus_{k=1}^{t}M_{k}$ and $\bigoplus_{k=1}^{s}N_{k}$ respectively. If $M
\cong N$ then $t=s$ and for any $k$ we have $M_{k} \cong N_{k}$ (and in
particular $v(M_{k})=v(N_{k})$). \label{isocomp}
\end{thm}

\begin{proof}
The ranks of $M$ and $N$ must be equal, and we apply induction on this common
rank. The case of rank 1 is immediate. Let $y \in M_{1}$ be a basis element as
in Corollary \ref{diag}. Then $v(y^{2})$ is minimal in $M$, and let $w \in N$ be
an element having the same norm as $y$ (such $w$ exists since $M \cong N$).
$v(w^{2})$ is also minimal in $N$, hence $v(N_{1})=v(M_{1})$. Write
$w=\sum_{k=1}^{s}w_{k}$ with $w_{k} \in N_{k}$ for any $1 \leq k \leq s$. The
minimality of $v(w^{2})$ implies $v(w_{k}^{2})>v(w^{2})$ for any $k\geq2$. Thus,
the image of $\frac{w^{2}}{w_{1}^{2}}$ modulo $I_{0}$ is 1. Since $v(2)=0$,
Lemma \ref{quadeq} yields the existence of $c \in R^{*}$ such that
$w^{2}=c^{2}w_{1}^{2}$, so that $z=cw_{1} \in N_{1}$ has the same norm as $y$.
Lemma \ref{sublat} allows us to write $M$ as $Ry\oplus(Ry)^{\perp}$ and
$N$ as $Rz\oplus(Rz)^{\perp}$. Let $L$ be a rank 1 lattice generated by an
element $x$ having the same norm as $y$ and $z$. The sublattices $(Ry)^{\perp}$
of $M$ and $(Rz)^{\perp}$ of $N$, together with this element $x$, satisfy the
conditions of Proposition \ref{cancr1} (indeed, $M \cong L\oplus(Ry)^{\perp}$
and $N \cong L\oplus(Rz)^{\perp}$ are isomorphic, and the valuation condition
holds by our choice of $y$ and $z$), so that $(Ry)^{\perp} \cong (Rz)^{\perp}$
by that proposition. Applying Lemma \ref{sublat} to $M_{1}$ with $y$ and to
$N_{1}$ with $z$ yields the decompositions $M_{1}=Ry \oplus
(Ry)^{\perp}_{M_{1}}$ and $N_{1}=Rz \oplus (Rz)^{\perp}_{N_{1}}$, where
the orthogonal complements $(Ry)^{\perp}_{M_{1}}$ and $(Ry)^{\perp}_{N_{1}}$ are
uni-valued with valuation $v(M_{1})=v(N_{1})$. Therefore we can write
$(Ry)^{\perp}=(Ry)^{\perp}_{M_{1}}\oplus\bigoplus_{k=2}^{t}M_{k}$ and
$(Rz)^{\perp}=(Rz)^{\perp}_{N_{1}}\oplus\bigoplus_{k=2}^{s}N_{k}$ as uni-valued
decompositions with increasing valuations, and the induction hypothesis implies
$t=s$, $M_{k} \cong N_{k}$ for $k\geq2$, and
$(Ry)^{\perp}_{M_{1}}\cong(Rz)^{\perp}_{N_{1}}$. As $Ry \cong Rz$ as well, we
deduce also $M_{1} \cong N_{1}$, which completes the proof of the theorem.
\end{proof}

The Witt Cancellation Theorem for 2-Henselian valuation rings in which 2 is
invertible follows as

\begin{cor}
Let $M$, $N$, and $L$ be three lattices over such a ring $R$. If $M \oplus L
\cong N \oplus L$ then $M \cong N$. \label{WittCan}
\end{cor}

\begin{proof}
Write $M=\bigoplus_{k=1}^{t}M_{k}$, $N=\bigoplus_{k=1}^{t}N_{k}$, and
$L=\bigoplus_{k=1}^{t}L_{k}$ as in Proposition \ref{decom}, where we assume that
$v(M_{k})=v(N_{k})=v(L_{k})$ for any $1 \leq k \leq t$ (allowing empty
components if necessary for such a unified notation). Then we can write $M
\oplus L$ and $N \oplus L$ as $\bigoplus_{k=1}^{t}(M_{k} \oplus L_{k})$ and
$\bigoplus_{k=1}^{t}(N_{k} \oplus L_{k})$, both being decompositions to
uni-valued lattices of increasing valuations. Theorem \ref{isocomp} yields
$M_{k} \oplus L_{k} \cong N_{k} \oplus L_{k}$ for any $1 \leq k \leq
t$, and we claim that $M_{k} \cong N_{k}$ for every such $k$. One way to
establish this assertion is by dividing the bilinear forms on $M_{k}$, $N_{k}$,
and $L_{k}$ by a suitable element of $R$ to make them unimodular, and then use
the Witt Cancellation Theorem for unimodular lattices proved in Theorem 4.4 of
\cite{[MH]}. Alternatively, $L_{k}$ has an orthogonal basis by Corollary
\ref{diag}, and we can ``cancel'' these basis elements iteratively using
Proposition \ref{cancr1} since the valuation condition is satisfied. In any
case, this assertion implies
$M=\bigoplus_{k=1}^{t}M_{k}\cong\bigoplus_{k=1}^{t}N_{k}=N$ as desired.
\end{proof}

Theorem \ref{isocomp} generalizes Theorem 4 of \cite{[D]} to (2-Henselian)
valuation rings of arbitrary rank, and Corollary \ref{WittCan} generalizes
Theorem 2 of \cite{[J1]} and Theorem 5 of \cite{[D]} to this case, again with
simplified proofs.

There are classical examples showing that without the condition $2 \in R^{*}$,
both Theorem \ref{isocomp} and Corollary \ref{WittCan} no longer hold. Over
$\mathbb{Z}_{2}$ the lattice $M$ generated by two orthogonal elements of norms
1 and 2 is isomorphic to the lattice $N$ having an orthogonal basis consisting
of vectors of norms 3 and 6. Since $v(1)=v(3)=0$ and $v(2)=v(6)=1$ but
$\frac{3}{1}=\frac{6}{2}=3\notin(\mathbb{Z}_{2}^{*})^{2}$, this example
demonstrates that Theorem \ref{isocomp} fails for $R=\mathbb{Z}_{2}$ (for the
analysis of $\mathbb{Z}_{2}$-lattices, see \cite{[J2]}). Still over
$\mathbb{Z}_{2}$, taking $x^{2}=y^{2}=0$ and $(x,y)=z^{2}=t=1$ in the proof of
Corollary \ref{diag} shows that adding an orthogonal norm 1 vector to the
hyperbolic plane generated by two norm 0 vectors $x$ and $y$ with $(x,y)=1$
(denoted $M_{0,0}$ in the notation of Sections \ref{Uni2} and \ref{Inv}) one
obtains a lattice admitting an orthonormal basis consisting of three elements
with norms 1, 1, and $-1$ respectively, which gives a counter-example to
Corollary \ref{WittCan} over $\mathbb{Z}_{2}$. The 2-Henselianity is also
important: Consider the ring $\mathbb{Z}_{p\mathbb{Z}}$ obtained from
$\mathbb{Z}$ by localizing in the prime ideal $p\mathbb{Z}$ for some odd prime
$p$. It is a discrete valuation ring with quotient field $\mathbb{F}_{p}$ of odd
characteristic, which is not complete. As in the first example over
$\mathbb{Z}_{2}$, the lattice generated by two orthonormal vectors of norms 1
and $p$ also admits an orthogonal basis whose norms are $1+p$ and $p+p^{2}$,
and unless $1+p$ is a square in $\mathbb{Z}$, this shows that Theorem
\ref{isocomp} does not hold over $\mathbb{Z}_{p\mathbb{Z}}$ (in case $1+p$ is a
square, one may use a similar argument with $1+t^{2}p$ for other
$t\in\mathbb{Z}$ fur this purpose). As for the general Witt Cancellation Theorem
for valuation rings which are not 2-Henselian but with $2 \in R^{*}$, finding
counter-examples seems complicated, in view of Proposition \ref{cancr1} and the
fact that the Witt Cancellation Theorem holds in general when when $M$, $N$, and
$L$ all have rank 1. We leave this question for further research.

\medskip

Theorem \ref{isocomp} shows that the classification of general $R$-lattices (for
$R$ a 2-Henselian valuation ring with $2 \in R^{*}$) reduces to the
classification of unimodular $R$-lattices in the following sense. For every
$v\geq0$ fix some element $\sigma_{v} \in R$ with valuation $v$, with
$\sigma_{0}=1$. By Theorem \ref{isocomp} every $R$-lattice $M$ can be written
\emph{uniquely} up to an isomorphism as $\bigoplus_{v}M_{v}(\sigma_{v})$ with
$M_{v}$ unimodular such that $M_{v}=0$ for all but finitely many $v$. Moreover,
in this case the unimodular lattices are determined up to isomorphism by their
(non-degenerate) restriction to the residue field $\mathbb{F}$ of $R$ (see, for
example, Theorem \ref{approxiso}---note that $v(2)=0$ by our assumption on $R$
and $v(M_{t})=0$ by unimodularity). Hence, the classification of $R$-lattices
simplifies to the classification of lattices over the field $\mathbb{F}$, whose
characteristic differs from $2$, for which many methods have been developed. For
general fields this problem is not at all simple: For example, for $\mathbb{F}$
a global field the isomorphism classes depend on all the completions of
$\mathbb{F}$. However, if $\mathbb{F}$ is \emph{finite} then the isomorphism
class of an $\mathbb{F}$-lattice $M$ is determined by its rank and sign (i.e,
the image of the determinant of a Gram
matrix of a basis of $M$ in the order 2 group
$\mathbb{F}^{*}/(\mathbb{F}^{*})^{2}$)---see, for example, Proposition 5 in
Chapter IV of \cite{[Se1]}. Let $1^{\varepsilon n}$ denote a unimodular
$R$-lattice ($R$ as above, with a finite residue field $\mathbb{F}$) whose
restriction modulo $I_{0}$ has rank $n$ and sign
$\varepsilon\in\mathbb{F}^{*}/(\mathbb{F}^{*})^{2}\cong\{\pm1\}$. Using the
shorthand $\sigma_{v}^{\varepsilon n}$ for $1^{\varepsilon n}(\sigma_{v})$, we
have thus proved

\begin{prop}
Any isomorphism class of lattices over a 2-Henselian valuation ring $R$ with
finite residue $\mathbb{F}$ of odd characteristic contains a unique
representative of the form $\bigoplus_{v}\sigma_{v}^{\varepsilon_{v}n_{v}}$
(where the direct sum is finite). \label{symbol}
\end{prop}

\begin{proof}
The existence of such a representative follows from Proposition \ref{decom} and
the fact that every uni-valued $R$-lattice is (up to isomorphism) of the form
$\sigma_{v}^{\varepsilon_{v}n_{v}}$ for unique $v$, $n_{v}$, and
$\varepsilon_{v}$ (see the previous paragraph). The uniqueness follows from
Theorem \ref{isocomp}. This proves the proposition.
\end{proof}

In the case where $R$ is a \emph{discrete} valuation ring we can take
$\sigma_{v}$ for $v\in\mathbb{N}$ to by the $v$th power of a uniformizer $\pi$
of $R$. In particular, Proposition \ref{symbol} yields the known symbols for
lattices over $p$-adic rings (for odd $p$), but it holds in much greater
generality.

\section{Unimodular Rank 2 Lattices \label{Uni2}}

Corollary \ref{diag} implies that there are only two types of lattices over a
valuation ring $R$ which are \emph{indecomposable} (namely, cannot be written
as the orthogonal direct sum of two proper sublattices). Such a lattice is
either generated by one element, or is spanned by two elements $x$ and $y$ such
that $v(x,y)$ is strictly smaller than the valuations of the norms of both $x$
and $y$. Both such lattices are uni-valued, hence (after fixing $a_{v}$ for each
possible valuation $v$) it suffices for the description of the isomorphism
classes of such lattices to consider just unimodular lattices. For rank 1
lattices the task is easy: Each isomorphism class of unimodular rank 1 lattices
corresponds to an element of $R^{*}/(R^{*})^{2}$, which is the norm of a
generator of a lattice in this isomorphism class. We consider classes modulo
$(R^{*})^{2}$ since multiplying the generator by $c \in R^{*}$ yields a
generator for the same module, with norm multiplied by $c^{2}$. In fact,
if the unimodularity assumption is relaxed, the isomorphism classes of
non-degenerate rank 1 lattices correspond to classes in
$(R\setminus\{0\})/(R^{*})^{2}$, and these assertions hold over any integral
domain $R$.

\smallskip

We now consider unimodular rank 2 lattices, in which the basis elements $x$ and
$y$ satisfy $(x,y) \in R^{*}$. In fact, every unimodular rank 2 lattice $L$
over a valuation ring admits such a basis: If for a given basis $x$ and $y$ of
$L$ we have $(x,y) \notin R^{*}$ (i.e., $(x,y) \in I_{0}$) then without loss of
generality unimodularity implies $v(x^{2})=0$. Thus, $x$ and $x+y$ form a basis
for $L$ such that $(x,x+y) \in R^{*}$. Multiplying $x$ or $y$ by an element of
$R^{*}$, we may assume $(x,y)=1$.

Given $\alpha$ and $\beta$ in $R$, we denote the rank 2 lattice spanned by
elements $x$ and $y$ with $x^{2}=\alpha$, $(x,y)=1$, and $y^{2}=\beta$ by
$M_{\alpha,\beta}$. Without loss of generality, we always assume $v(\alpha) \leq
v(\beta)$. An interesting question, which will be answered under some
assumptions in Section \ref{Inv} below, is finding necessary and sufficient
conditions on $\alpha$, $\beta$, $\gamma$ and $\delta$ such that
$M_{\alpha,\beta} \cong M_{\gamma,\delta}$. The present Section is devoted to
the description of the lattices $M_{\alpha,\beta}$, and is divided into three
subsections. In Subsection \ref{IsoGen} we classify the lattices $M_{\alpha,0}$
and prove a technical lemma which will later have various applications.
Subsection \ref{MaxVal} considers conditions for existence of a primitive
element with norm of maximal valuation in $M_{\alpha,\beta}$. Finally,
Subsection \ref{GenArfInv} defines the generalized Arf invariant of such a
lattice (under this maximality assumption on $\beta$) and proves that it is an
invariant of the isomorphism class of the lattice. Unless stated otherwise, $R$
is a 2-Henselian valuation ring.

\subsection{Isotropic Lattices and General Technicalities \label{IsoGen}}

Recall that a non-zero norm 0 vector is called \emph{isotropic}, and a lattice
is called \emph{isotropic} if it contains an isotropic vector. Our first
observation is

\begin{lem}
If $v(\alpha)+v(\beta)>2v(2)$ then $M_{\alpha,\beta}$ is isotropic. \label{iso}
\end{lem}

\begin{proof}
Write $(y+tx)^{2}=\alpha t^{2}+2t+\beta$. Then the coefficients $A=\alpha$,
$B=2$, and $C=\beta$ satisfy the condition of Lemma \ref{quadeq}. Hence there
exists some $t\in\mathbb{K}$ which annihilates this expression, and the
inequality $v(C)>v(B)$ (which follows from $v(\alpha) \leq v(\beta)$ and
$v(\alpha)+v(\beta)>2v(2)$) implies that we can take $t \in R$ (and even $t \in
I_{0}$). The vector $y+tx$ of $M_{\alpha,\beta}$ is then isotropic, which
proves the lemma.
\end{proof}

We now prove another assertion about the possible values of norms of
elements of $M_{\alpha,\beta}$ having minimal valuation under certain
conditions. In the case $v(2)>0$ the \emph{Artin--Schreier map}
$\rho:\mathbb{F}\to\mathbb{F}$ is defined by $\rho(x)=x^{2}-x$. It is an
additive homomorphism on $\mathbb{F}$, whose kernel is the prime subfield
$\mathbb{F}_{2}\subseteq\mathbb{F}$, and its image is denoted
$\mathbb{F}_{AS}$. By some abuse of notation, the map from $R$ to $R$ defined by
the same formula $x \mapsto x^{2}-x$ will be also denoted $\rho$, though it is
no longer a homomorphism of additive groups. We denote $\rho(R)$ by $R_{AS}$.
First we need

\begin{lem}
If $y \in R$ then $y \in R_{AS}$ holds if and only if
$y+I_{0}\in\mathbb{F}_{AS}$. \label{RASFAS}
\end{lem}

\begin{proof}
If $y=\rho(x)$ then $y+I_{0}=\rho(x+I_{0})$. Conversely, if
$y+I_{0}=\rho(x+I_{0})$ for some $x \in R$ then substitute $t=s+x$ in the
polynomial $f(t)=t^{2}-t-y$ in order to obtain $g(s)=s^{2}-(1-2x)s+(x^{2}-x-y)$.
The coefficients $A=1$, $B=2x-1$, and $C=x^{2}-x-y$ satisfy $v(AC)>v(B^{2})$
since $v(C)>0$ and $v(A)=v(B)=0$ (recall that $2 \in I_{0}$). Hence Lemma
\ref{quadeq} yields a root $s \in I_{0} \subseteq R$ of this polynomial. The
element $x+s$ of $R$ (with the same $\mathbb{F}$-image as $x$) satisfies
$\rho(x+s)=y$, which completes the proof of the lemma.
\end{proof}

We now prove an important technical result which will be needed later.

\begin{lem}
Let $\alpha$ and $\beta$ be two elements of $R$ such that $v(\beta) \geq v(2)$
and $v(\beta) \geq v(\alpha)$. We define the set $T$ to be
$(R^{*})^{2}\cdot(\alpha+2R)$ in case $v(\beta)>v(2)$ and as
$(R^{*})^{2}\cdot\big(\alpha+\frac{4}{\beta}R_{AS}\big)$ if $v(\beta)=v(2)$.
$(i)$ If $v(\alpha) \geq v(2)$ and $v(\beta)>v(2)$ then $T=2R$, and this is the
set of all norms of elements in $M_{\alpha,\beta}$. $(ii)$ In the case
$v(2)>v(\alpha)$ the set $T$ consists of all the norms of elements of
$M_{\alpha,\beta}$ having minimal valuation, which then equals $v(\alpha)$.
$(iii)$ In the boundary case $v(\beta)=v(2)=v(\alpha)$ an
element lies in $T$ precisely when it is the norm of a primitive element of
$M_{\alpha,\beta}$. \label{norm2}
\end{lem}

\begin{proof}
First we show that an element of $R$ lies in $T$ if and only if it is the norm
of some element $z \in M_{\alpha,\beta}$ of the form $z=cx+dy$ with $c \in
R^{*}$. Indeed, such an element of $M_{\alpha,\beta}$ can be written as
$c(x+sy)$ with $s \in R$, and its norm $c^{2}(\alpha+2s+s^{2}\beta)$ is of the
form $c^{2}(\alpha+2r)$ since $2|\beta$. Moreover, if $v(\beta)=v(2)$ then by
writing $s=-\frac{2}{\beta}x$ we find that
$r=\frac{2}{\beta}\rho(x)\in\frac{2}{\beta}R_{AS}$. Conversely, given $r$ and
$c$ we need to show that $c^{2}(\alpha+2r)$ can be obtained as the norm of such
$z \in M_{\alpha,\beta}$. By writing $z=c(x+sy)$ again this assertion reduces to
finding $s \in R$ such that $r=s+s^{2}\frac{\beta}{2}$. Consider the polynomial
$f(s)=\frac{\beta}{2}s^{2}+s-r$. If $v(\beta)>v(2)$ and $r$ is arbitrary, then
the coefficients $A=\frac{\beta}{2} \in I_{0}$, $B=1 \in R^{*}$, and $C=-r \in
R$ satisfy the conditions of Lemma \ref{quadeq}, yielding a solution $s \in R$
(of valuation $v\big(\frac{C}{B}\big)=v(r)\geq0$). On the other hand, if
$v(\beta)=v(2)$ and $r\in\frac{2}{\beta}R_{AS}$ then the substitution
$s=-\frac{2}{\beta}t$ takes $f(s)$ to the form
$\frac{2}{\beta}\big(t^{2}-t+\frac{\beta}{2}r\big)$, which has a root $t$
by our assumption on $r$.

Now, if $v(\beta)>v(2)$ and $v(\alpha) \geq v(2)$ then $T=2R$, and since for
every element $z=ax+by \in M_{\alpha,\beta}$ the three terms $a^{2}\alpha$,
$2ab$, and $b^{2}\beta$ of $z^{2}$ are divisible by 2, we obtain $T=\{z^{2}|z
\in M_{\alpha,\beta}\}$. This proves $(i)$. On the other hand, assume
$v(\alpha)<v(2)$, and let $z \in M_{\alpha,\beta}$ be an element whose norm has
the same valuation as $\alpha$. We write $z=cx+dy$, and if $c \in I_{0}$ then
the three terms $c^{2}\alpha$, $2cd$, and $d^{2}\beta$ of $(cx+dy)^{2}$ lie in
$I_{v(\alpha)}$. Since this contradicts the assumption $v(z^{2})=v(\alpha)$ we
deduce $c \in R^{*}$, and we have already seen that $z^{2} \in T$ for such $z$.
This proves $(ii)$. It remains to consider the case $v(\beta)=v(2)=v(\alpha)$.
In this case a primitive element of $M_{\alpha,\beta}$ not considered in the
above paragraph takes the form $z=h(y+tx)$ with $h \in R^{*}$ and $t \in I_{0}$,
and satisfies $z^{2}=h^{2}(\beta+2t+t^{2}\alpha)$. But the element
$w=h\big[\big(\frac{2}{\alpha}+t\big)x-y\big]$ has the same norm as $z$ and the
coefficient of $x$ in $w$ is in $R^{*}$ (since $\frac{2}{\alpha} \in R^{*}$ and
$t \in I_{0}$), so that the norm of $z$ lies in $T$ for such $z$ as well. This
proves $(iii)$ and completes the proof of the lemma.
\end{proof}

We remark that the element $w \in M_{\alpha,\beta}$ defined at the end of the
proof of Lemma \ref{norm2} is the image of $z$ under the reflection with respect
to $x$, taking $u \in M_{\alpha,\beta}$ to $u-\frac{2(u,x)}{\alpha}x$. This
element is well-defined as an automorphism of $M_{\alpha,\beta}$ since
$\frac{2}{\alpha} \in R$, though it is not a reflection with respect to a
decomposition since $M_{\alpha,\beta}$ does not decompose as
$Rx\oplus(Rx)^{\perp}$ if $v(\alpha)>0$.

The proof of parts $(i)$ and $(ii)$ of Lemma \ref{norm2} allows us to classify
isotropic rank 2 unimodular lattices over any valuation ring (not necessarily
2-Henselian):

\begin{prop}
Let $R$ be any valuation ring. The lattices $M_{\alpha,0}$ and $M_{\gamma,0}$
are isomorphic if and only if $\gamma=c^{2}(\alpha+2r)$ for some $c \in R^{*}$
and $r \in R$. Therefore the isomorphism classes of isotropic unimodular rank 2
lattices are in one-to-one correspondence with the set $(R/2R)/(R^{*})^{2}$.
\label{r2rep0}
\end{prop}

\begin{proof}
The only place where we used the 2-Henselian property of $R$ in the proof of
Lemma \ref{norm2} was in our search for a solution to the equation
$\frac{\beta}{2}s^{2}+s=r$ for $r \in R$. But if $\beta=0$ then $s=r$ is a
solution, so that Lemma \ref{norm2} with $\beta=0$ holds over any valuation ring
$R$. Thus, if $x$ and $y$ form the basis for $M_{\alpha,0}$ as above then
$c(x+ry)$ and $\frac{y}{c}$ span the same lattice and define an isomorphism with
$M_{\gamma,0}$ for $\gamma=c^{2}(\alpha+2r)$. It remains to show that if
$M_{\alpha,0} \cong M_{\gamma,0}$ then $\gamma=c^{2}(\alpha+2r)$ for some $r \in
R$ and $c \in R^{*}$. If $2|\alpha$ then $\gamma$, as the norm of some element
of $M_{\alpha,0}$, is divisible by 2 (the proof of Lemma \ref{norm2} again),
which completes the proof for this case. Assume now $v(\alpha)<v(2)$. The
isomorphism from $M_{\gamma,0}$ to $M_{\alpha,0}$ takes the isotropic generator
of $M_{\gamma,0}$ to a primitive isotropic vector $w \in M_{\alpha,0}$, which
can be either $\frac{y}{c}$ or $\frac{1}{c}(y-\frac{2}{\alpha}x\big)$ for $c \in
R^{*}$. The other basis vector of $M_{\gamma,0}$ must be taken to some $z \in
M_{\alpha,0}$ with $(z,w)=1$. In the first case we have $z=c(x+ry)$ for some $r
\in R$, and the norm of $z$ is $c^{2}(\alpha+2r)$ as shown above. Otherwise
$w=\frac{1}{c}\big(y-\frac{2}{\alpha}x\big)$, and writing $z=ax+by$ with $a$
and $b$ in $R$, the equality $(z,w)=1$ implies $a=-c-\frac{2}{\alpha}b$.
It follows that $z$ takes the form $-cx+b\big(y-\frac{2}{\alpha}x\big)$ for some
$b \in R$, and its norm $c^{2}\alpha-2bc$ also has the asserted form. This
proves the proposition.
\end{proof}

The reflection with respect to $x$ appears implicitly also in the proof of
Proposition \ref{r2rep0}, since it takes the element $\frac{y}{c}$ of
$M_{\alpha,0}$ to $\frac{1}{c}\big(y-\frac{2}{\alpha}x\big)$.

We remark that if $\alpha \in R^{*}$ then $M_{\alpha,0}$ is decomposable, since
the elements $x$ and $t=x-\alpha y$ are orthogonal and have the norms $\alpha$
and $-\alpha$ respectively. Conversely, a direct sum of two unimodular rank 1
lattices which is isotropic must be of this form: If $z$ and $w$ are
perpendicular and have norms in $R^{*}$ then for some combination $az+bw$ to
have be isotropic we must have $a^{2}z^{2}=-b^{2}w^{2}$. Hence $v(a)=v(b)$, and
by replacing $w$ by $u=\frac{b}{a}w$ we obtain a generator $u$ for $Rw$ such
that $u^{2}=-z^{2}$. Therefore Proposition \ref{r2rep0} implies the following

\begin{cor}
For $\alpha \in R^{*}$ ($R$ being any valuation ring) denote $H_{\alpha,0}$ the
lattice generated by two orthogonal elements of norms $\alpha$ and $-\alpha$.
Given $\alpha$ and $\gamma$ in $R^{*}$, the relation $H_{\alpha,0} \cong
H_{\gamma,0}$ holds if and only if $\gamma=c^{2}\alpha+2r$ for $c \in R^{*}$ and
$r \in R$. \label{hypiso}
\end{cor}

\begin{proof}
The previous paragraph shows that for any $\alpha \in R^{*}$ the lattices
$H_{\alpha,0}$ and $M_{\alpha,0}$ are isomorphic. Hence the assertion follows
from Proposition \ref{r2rep0}.
\end{proof}

In particular, if $v(2)=0$ then $R/2R$ is trivial, and Proposition \ref{r2rep0}
implies that every isotropic unimodular rank 2 lattice over $R$ is a hyperbolic
plane (namely, a lattice isomorphic to $M_{0,0}$ in the notation of Proposition
\ref{r2rep0}). This statement is in correspondence with the results at the end
of Section \ref{Unique}, since the same assertion holds over $\mathbb{F}$ if
the characteristic of $\mathbb{F}$ is not 2. Corollary \ref{hypiso} implies in
this case that all the lattices of the form $H_{\alpha,0}$ with $\alpha \in
R^{*}$ are isomorphic. On the other hand, if $v(2)>0$ then the elements of
$(I_{0}/2R)/(R^{*})^{2}$ correspond to \emph{indecomposable} isotropic
unimodular lattices of rank 2. If $v(2)=0$ then $I_{0}/2R$ is not well-defined
and can be considered as the empty set (since $I_{0}$ is the complement of
$R^{*}$ and the image of $R^{*}$ in $R/2R$ is the entire set $R/2R$), which
corresponds to the fact that there exist no indecomposable rank 2 lattices in
this case (Corollary \ref{diag} again).

\subsection{Primitive Vectors with Norms of Maximal Valuation \label{MaxVal}}

The anisotropic case is more delicate. Lemma \ref{iso} allows us to restrict
attention to the case $v(\alpha)+v(\beta) \leq 2v(2)$ when we consider
anisotropic lattices. The following lemma will turn out useful in what follows.

\begin{lem}
If $v(\alpha)+v(\beta)\leq2v(2)$ then the valuation of the norm of a non-zero
element of $M_{\alpha,\beta}$ of the form $tx+sy$ with $t$ and $s$ in $R$ can
be evaluated as $\min\{v(t^{2}\alpha),v(s^{2}\beta)\}$, provided that the two
terms $t^{2}\alpha$ and $s^{2}\beta$ have different valuations. In case
$v(t^{2}\alpha)=v(s^{2}\beta)$ and $v(\alpha)+v(\beta)<2v(2)$, the valuation of
$(tx+sy)^{2}$ is larger than the common valuation of $t^{2}\alpha$ and
$s^{2}\beta$ if and only if $v(t^{2}\alpha+s^{2}\beta) \geq
v(t^{2}\alpha)=v(s^{2}\beta)$. \label{vnorm}
\end{lem}

\begin{proof}
The element in question has norm $t^{2}\alpha+2st+s^{2}\beta$, and we claim that
$v(2st)>\min\{v(t^{2}\alpha),v(s^{2}\beta)\}$ under our assumptions. If $s=0$
this is clear, so assume $s\neq0$. Now, assuming by contradiction that $v(2st)
\leq v(s^{2}\beta)$ and $v(2st) \leq v(t^{2}\alpha)$, we obtain the inequalities
$v\big(\frac{t}{s}\big) \leq v\big(\frac{\beta}{2}\big)$ and
$v\big(\frac{t}{s}\big) \geq v\big(\frac{2}{\alpha}\big)$, the combination of
which yields $v\big(\frac{\beta}{2}\big) \geq v\big(\frac{2}{\alpha}\big)$. But
the latter inequality implies $v(\alpha)+v(\beta)\geq2v(2)$, and this may occur
with $v(\alpha)+v(\beta)\leq2v(2)$ only if $v(\alpha)+v(\beta)=2v(2)$ and
$v(t^{2}\alpha)=v(s^{2}\beta)$, a case which we have excluded in our
assumptions. This establishes the claim. It follows that comparing
$v\big((sx+ty)^{2}\big)$ with $\min\{v(t^{2}\alpha),v(s^{2}\beta)\}$ is the same
as comparing $v(t^{2}\alpha+s^{2}\beta)$ with that minimum, which completes the
proof of the lemma.
\end{proof}

If $v(\alpha) \leq v(2)$ (and $v(\alpha) \leq v(\beta)$, as always), then $x$
is a primitive element of $M_{\alpha,\beta}$ whose norm has minimal valuation.
This is obvious, since the three terms appearing in the expansion of the norm
of any element $ax+by$ of $M$ have valuations of at least $v(\alpha)$. Maximal
valuation is a more complicated property, whose existence is guaranteed only
under some conditions on $R$ in Proposition \ref{maxv} below. We shall restrict
attention, from now on, only to those lattices $M_{\alpha,\beta}$ which do
contain such elements. Henceforth, \emph{we assume in our notation
$M_{\alpha,\beta}$ that $\beta$, as a norm of a primitive element of
$M_{\alpha,\beta}$, has maximal valuation}. For the characterization of such
lattices, we need a parity notion for elements of the value group $\Gamma$. We
call an element of $\Gamma$ \emph{even} if it is divisible by 2 in $\Gamma$
(namely, if it is the valuation of an element of $(\mathbb{K}^{*})^{2}$), and
\emph{odd} otherwise. A distinguished class of elements of even valuation is
given in the following
\begin{defn}
An element $a$ of even valuation is called an \emph{approximate square} if $a
\equiv b^{2}(\mathrm{mod\ }I_{v(a)})$ for some $b \in R$. If $v(a)=0$ we call
$a$ a \emph{residual square}. Given an even valuation $u$, we let $J_{u}$ be the
union of the ideal $I_{u}$ with the set of approximate squares of valuation $u$.
\label{appsq}
\end{defn}
The condition that $v(a)$ is even in Definition \ref{appsq} is in fact
redundant, since no element of odd valuation can satisfy the required property
for being an approximate square. We claim that if $v(2)>0$ then the sets $J_{u}$
for even $u$ are additive subgroups of $R$. Indeed, $J_{0}$ is just the inverse
image of $\mathbb{F}^{2}$ (including 0) under the projection from $R$ to
$\mathbb{F}$; It is a group since $x \mapsto x^{2}$ is additive on a field of
characteristic 2 and $\mathbb{F}^{2}$ is just the image of this map. In the
general case we observe that if $a \in R$ has even valuation $u$ then $a$ is an
approximate square if and only if $\frac{a}{\sigma^{2}}$ is a residual square
for some (hence any) $\sigma \in R$ with $2v(\sigma)=u$. Hence
$J_{u}=\sigma^{2}J_{0}$ for any such $\sigma$, showing that it is also a group.
As $I_{u} \subseteq J_{u}$, we may allow ourselves the slight abuse of notation
by referring as approximate squares also to images in $R/I_{u}$ of approximate
squares of valuation $u$, and this remains well-defined. We remark that the
natural definition of $J_{u}$ for an odd valuation $u$ is just $I_{u}$, since
there are no approximate squares of valuation $u$.

In our examination of lattices $M_{\alpha,\beta}$ with primitive vectors of
norms with maximal valuation we shall distinguish among the cases
$v(\alpha)+v(\beta)<2v(2)$ and $v(\alpha)+v(\beta)=2v(2)$.

\begin{prop}
The case $v(\alpha)+v(\beta)<2v(2)$ and $v(\beta)$ is maximal occurs precisely
when the element $\frac{\beta}{\alpha}$ of $R$ is not an approximate square.
More explicitly, if $v(\alpha)+v(\beta)<2v(2)$ then $v(\beta)$ is maximal either
when $v\big(\frac{\beta}{\alpha}\big)$ is odd, or when
$v\big(\frac{\beta}{\alpha}\big)$ is even but
$\frac{\beta}{\alpha\sigma^{2}}+I_{0}$ is not in $(\mathbb{F}^{*})^{2}$ for some
(hence any) $\sigma \in R$ with $2v(\sigma)=v\big(\frac{\beta}{\alpha}\big)$,
but in no other case. \label{v<v4}
\end{prop}

\begin{proof}
A primitive element $z \in M_{\alpha,\beta}$ takes either the form $c(y+tx)$
with $t \in R$ and $c \in R^{*}$ or the form $c(x+sy)$ with $s \in I_{0}$ and $c
\in R^{*}$. For the valuation of the norm of $z$, we can assume $c=1$. For the
norm of an element of the form $x+sy$, the fact that $v(\beta) \geq v(\alpha)$
and $s \in I_{0}$ allows us to apply Lemma \ref{vnorm}, which yields that the
valuation of this norm is just $v(\alpha)$. As for $z=y+tx$, Lemma \ref{vnorm}
shows that the valuation of $z^{2}$ is $\min\{v(\beta),v(t^{2}\alpha)\}$ unless
$v(\beta)=v(t^{2}\alpha)$ and the sum $\beta+t^{2}\alpha$ has larger valuation.
If $v\big(\frac{\beta}{\alpha}\big)$ is odd then the equality
$v(\beta)=v(t^{2}\alpha)$ cannot be achieved, hence any primitive $z$ has norm
of valuation at most $v(\beta)$. Otherwise we take $t=\sigma r$ for $\sigma$ as
above and $r \in R^{*}$, and $\beta+t^{2}\alpha$ has valuation larger than
$v(\beta)$ if and only if $1+r^{2}\cdot\frac{\alpha\sigma^{2}}{\beta} \in
I_{0}$. The maximality of $v(\beta)$ is thus equivalent to
$1+r^{2}\cdot\frac{\alpha\sigma^{2}}{\beta}$ being in $R^{*}$ for every such
$r$, and since 2 must be in $I_{0}$ to allow $v(\alpha)+v(\beta)<2v(2)$, the
latter condition forbids the image of $\frac{\beta}{\alpha\sigma^{2}}$ modulo
$I_{0}$ to belong to $(\mathbb{F}^{*})^{2}$. This proves the proposition.
\end{proof}

It is clear that Proposition \ref{v<v4} can be stated in terms of properties of
the product $\alpha\beta$ rather than the quotient $\frac{\beta}{\alpha}$.

\begin{cor}
If the residue field $\mathbb{F}$ of $R$ is perfect then the situation described
in Proposition \ref{v<v4} occurs only if $v\big(\frac{\beta}{\alpha}\big)$ is
odd. \label{v<v4perf}
\end{cor}

\begin{proof}
$\mathbb{F}$ is of characteristic 2, and it is perfect if and only if
$(\mathbb{F}^{*})^{2}=\mathbb{F}^{*}$. Hence the second setting in Proposition
\ref{v<v4} cannot occur in this case.
\end{proof}

Using Proposition \ref{v<v4}, we derive a condition on $R$ assuring the
existence of a primitive element whose norm has maximal valuation in every
$R$-lattice $M_{\alpha,\beta}$. We recall that an extension $\mathbb{L}$ of
$\mathbb{K}$ with a valuation $w$ on $\mathbb{L}$ such that
$w\big|_{\mathbb{K}}=v$ is called \emph{immediate} if $w(\mathbb{L})=\Gamma$ and
the quotient field $S/J_{0}$, with $S$ the valuation ring of $(\mathbb{L},w)$
and $J_{0}$ the maximal ideal in $S$, is isomorphic to $\mathbb{F}$.

\begin{prop}
Assume that $\mathbb{K}$ admits no quadratic immediate extensions. Then every
lattice $M_{\alpha,\beta}$ contains a primitive element whose norm has maximal
valuation. \label{maxv}
\end{prop}

\begin{proof}
Let $M_{\alpha,\beta}$ be a lattice without such an element. First we observe
that $v(\alpha)+v(\beta)<2v(2)$. For if $v(\alpha)+v(\beta)>2v(2)$ then
$M_{\alpha,\beta}$ is isotropic by Lemma \ref{iso}, and the norm of an
isotropic vector has maximal valuation $\infty$. Moreover, if
$v(\alpha)+v(\beta)=2v(2)$ then the fact the $\beta$ is not maximal allows us
to find a primitive element of norm $\delta$ with $v(\delta)>v(\beta)$. This
implies $M_{\alpha,\beta} \cong M_{\gamma,\delta}$ for some $\gamma$ with
$v(\gamma) \geq v(\alpha)$, and we are again in the isotropic case.

Now, a primitive element of $M_{\alpha,\beta}$ whose norm has valuation larger
than $v(\alpha)$ must be of the form $z=c(y+tx)$ for some $c \in R^{*}$ (see the
second paragraph of the proof of Lemma \ref{norm2}), and again we can take $c=1$
since we are interested only in $v(z^{2})$. We now construct a sequence of
elements $z_{\sigma}=y+t_{\sigma}x$, for $\sigma$ in some maximal well-ordered
set $\Sigma$, whose norms $\beta_{\sigma}=z_{\sigma}^{2}$ satisfy
$v(\beta_{\tau})>v(\beta_{\sigma})$ for $\tau>\sigma$. We do this using
transfinite induction, starting with $t_{0}=0$, $z_{0}=y$, and
$\beta_{0}=\beta$. Assume that we constructed $z_{\sigma}$ for
$\sigma\in\Sigma$. If $\Sigma$ has a maximal element $\tau$, then we can find
some $t_{\tau+1}$ such that $z_{\tau+1}$ has norm $\beta_{\tau+1}$ with
valuation bigger than $v(\beta_{\tau})$ (since $\beta_{\tau}$ is not maximal).
Then the index $\tau+1$ increases $\Sigma$. If $\Sigma$ does not contain any
maximal element, but there exists some primitive element of $M_{\alpha,\beta}$
having norm with valuation exceeding $v(\beta_{\sigma})$ for every
$\sigma\in\Sigma$, then we saw that this element can be written as
$z_{\tau}=y+t_{\tau}x$, and we increase $\Sigma$ by adding $\tau$ as a new
maximal element. We stop when it is impossible to add further elements to
$\Sigma$. The fact that we take elements from a fixed lattice implies that this
transfinite process must terminate. We thus obtain a well-ordered set $\Sigma$,
having no last element, and $t_{\sigma}$ for each $\sigma\in\Sigma$ such that
$v(\beta_{\tau})>v(\beta_{\sigma})$ for every $\tau>\sigma$ in $\Sigma$, and
such that no primitive element of $M_{\alpha,\beta}$ has norm with valuation
exceeding $v(\beta_{\sigma})$ for every $\sigma\in\Sigma$.

We claim that $\{t_{\sigma}\}_{\sigma\in\Sigma}$ is an algebraic
pseudo-convergent sequence in $R$, with no pseudo-limit in $R$ (see Definitions
10, 12, and 15 in Section 2 of \cite{[Sch]}). Indeed, let $\sigma$ and $\tau$ be
elements of $\Sigma$ with $\sigma<\tau$, and write
$z_{\tau}=z_{\sigma}+(t_{\tau}-t_{\sigma})x$. As
$v(\alpha)+v(\beta_{\sigma})<2v(2)$, Lemma \ref{vnorm} implies that
$v(\beta_{\sigma})=v\big((t_{\tau}-t_{\sigma})^{2}\alpha\big)$, for otherwise
the condition $v(\beta_{\tau})>v(\beta_{\sigma})$ cannot be satisfied (the fact
that $(x,z_{\sigma})$ equals $1+t_{\sigma}\alpha$ rather than $1$ does not
affect the validity of Lemma \ref{vnorm}). As in the proof of Proposition
\ref{v<v4} we obtain that $v\big(\frac{\beta_{\sigma}}{\alpha}\big)$ is even
and equals $2v(t_{\tau}-t_{\sigma})$ for each $\sigma\in\Sigma$. Hence
$v(t_{\tau}-t_{\sigma})=\eta_{\sigma}$ with $\eta_{\sigma}\in\Gamma$ strictly
increasing with $\sigma$ (as
$2\eta_{\sigma}=v\big(\frac{\beta_{\sigma}}{\alpha}\big)$), which shows that
$\{t_{\sigma}\}_{\sigma\in\Sigma}$ is pseudo-convergent. Moreover,
$\beta_{\sigma}=f(t_{\sigma})$ for $f(t)=\beta+2t+\alpha t^{2}$, and since
$v(\beta_{\sigma})$ strictly increases with $\sigma$, the algebraicity of
$\{t_{\sigma}\}_{\sigma\in\Sigma}$ follows. Had this pseudo-convergent sequence
a pseudo-limit $s \in R$, a similar argument would show that $y+sx$ has norm
$\delta$ with $v(\delta) \geq v(\beta_{\sigma})$ for every $\sigma\in\Sigma$,
contrary to our assumption on $\Sigma$. Then Lemmas 12 and 19 of Section 2 of
\cite{[Sch]} yield a quadratic immediate extension $\mathbb{L}$ of
$\mathbb{K}$ generated by adding a pseudo-limit to this sequence, in
contradiction to our assumption on $\mathbb{K}$. This contradiction shows that
$M_{\alpha,\beta}$ must contain a primitive element with maximal valuation.
\end{proof}

In particular, Proposition \ref{maxv} shows that if $\mathbb{K}$ is
\emph{maximally complete} (i.e., admits no immediate extensions at all) then
every lattice $M_{\alpha,\beta}$ contains a primitive element of norm with
maximal valuation. The proof of Proposition \ref{maxv} also shows that any
lattice $M_{\alpha,\beta}$ contains such a primitive element in case the set of
positive $\gamma\in\Gamma$ which are smaller than $2v(2)$ is \emph{finite}
(e.g., when $\Gamma=\mathbb{Z}$), a fact which is also easily verified
directly.

\smallskip

The conditions for $v(\alpha)+v(\beta)=2v(2)>0$ are somewhat different.

\begin{prop}
If $v(\alpha)+v(\beta)=2v(2)$ then $v(\beta)$ is maximal if and only if
$\frac{\alpha\beta}{4} \in R^{*} \setminus R_{AS}$. \label{v=v4}
\end{prop}

\begin{proof}
The condition $v(\alpha)+v(\beta)=2v(2)$ implies that
$\varepsilon=-\frac{\alpha\beta}{4} \in R^{*}$. As in the proof of Proposition
\ref{v<v4}, we may restrict attention to norms of elements of the form $z=y+tx
\in M_{\alpha,\beta}$, and the norm of such an element has valuation either
$v(t^{2}\alpha)$ or $v(\beta)$ unless $v(t)=v\big(\frac{2}{\alpha}\big)$ (which
is equivalent here to $2v(t)=v\big(\frac{\beta}{\alpha}\big)$). We therefore
write $t=-\frac{2}{\alpha}s$ (with $s \in R^{*}$), and then
$z^{2}=\frac{4}{\alpha}(s^{2}-s-\varepsilon)$. As
$v\big(\frac{4}{\alpha}\big)=v(\beta)$, the valuation of $z^{2}$ exceeds
$v(\beta)$ if and only if $\varepsilon+I_{0}=\rho(s+I_{0})$. Since $z$ is
primitive and $s \in R$ is arbitrary, we find that $v(\beta)$ is maximal if and
only if $-\varepsilon+I_{0}=\varepsilon+I_{0}\notin\mathbb{F}_{AS}$, which is
equivalent, by Lemma \ref{RASFAS}, to the desired condition $-\varepsilon \notin
R_{AS}$. This proves the proposition.
\end{proof}

\begin{cor}
If $\rho$ is surjective (in particular if $\mathbb{F}$ is algebraically closed)
then any lattice $M_{\alpha,\beta}$ with $v(\alpha)+v(\beta)=2v(2)$ is
isotropic. \label{v=v4alg}
\end{cor}

\begin{proof}
The proof of Proposition \ref{v=v4} shows that $v(\beta)$ cannot be maximal,
since we can take $s \in R$ such that $v(s^{2}-s-\varepsilon)>0$. But this shows
that $M_{\alpha,\beta} \cong M_{\alpha,\delta}$ with $v(\delta)>v(\beta)$ (take
the basis $x$ and $\frac{y+tx}{1-2s}$ with $t=-\frac{2}{\alpha}s$ as in the
latter proof). Since $v(\alpha)+v(\delta)>2v(2)$, Lemma \ref{iso} shows that
this lattice is isotropic. This proves the corollary.
\end{proof}

\subsection{Generalized Arf Invariants \label{GenArfInv}}

Following Lemma \ref{iso} and Propositions \ref{v<v4} and \ref{v=v4}, we define,
under the assumption that $v(2)>0$, the following invariant of a unimodular rank
2 lattice containing a primitive element with norm of maximal
valuation.
\begin{defn}
Let $M_{\alpha,\beta}$ be a unimodular rank 2 lattice over $R$ such that $\beta$
has maximal valuation. We define the \emph{generalized Arf invariant} of
$M_{\alpha,\beta}$ as follows:
\begin{enumerate}
\item If $M_{\alpha,\beta}$ is isotropic (i.e., $\beta=0$), define the
generalized Arf invariant to be 0. This is called a \emph{vanishing} generalized
Arf invariant.
\item If $v(\alpha)+v(\beta)<2v(2)$ and is odd, then we define the generalized
Arf invariant of $M_{\alpha,\beta}$ to be the class
$\alpha\beta+I_{v(\alpha\beta)}$. These generalized Arf invariants are called
\emph{odd}.
\item In the case where $0<v(\alpha)+v(\beta)<2v(2)$ and is even, the
generalized Arf invariant of $M_{\alpha,\beta}$ is defined to be the (non-zero)
class $\alpha\beta+J_{v(\alpha\beta)}$ (modulo the group of approximate squares
of valuation $v(\alpha\beta)$). This type of generalized Arf invariants is
called \emph{even}.
\item In case $v(\alpha)+v(\beta)=2v(2)$ we take the generalized Arf invariant
of $M_{\alpha,\beta}$ to be the image of $\alpha\beta$ in $4R/4R_{AS}$. These
are called \emph{exact} generalized Arf invariants.
\end{enumerate}
\label{GAIdef}
\end{defn}
We remark that the odd and even cases in Definition \ref{GAIdef} may be unified,
since we saw that for odd $u$ the natural definition is $J_{u}=I_{u}$. For the
exact case, the group $4R/4R_{AS}$ in question is the image of
$\mathbb{F}/\mathbb{F}_{AS}$ (depending only on $\mathbb{F}$) arising from
multiplication by 4 on $R$ and projecting onto the appropriate quotient.

The importance of the generalized Arf invariant from Definition \ref{GAIdef} is
revealed in the following
\begin{prop}
The generalized Arf invariant is an invariant of the isomorphism class of
$M_{\alpha,\beta}$. \label{Arf}
\end{prop}

\begin{proof}
The valuations of $\alpha$ and $\beta$ are well-defined by the minimality and
maximality assumptions. Denote $v(\alpha)+v(\beta)$ by $u$. Recall that the
discriminant of $M_{\alpha,\beta}$ is $-1+\alpha\beta$, and we consider its
class modulo $(R^{*})^{2}$. If $M_{\gamma,\delta} \cong M_{\alpha,\beta}$ then
$1-\alpha\beta=c^{2}(1-\gamma\delta)$, and since $\beta$ and $\delta$ lie in
$I_{0}$, it follows that $c^{2}\in1+I_{0}$ hence $c\in1+I_{0}$. If $u>2v(2)$
then Lemma \ref{iso} implies $u=\infty$, $\beta=\delta=0$, both discriminants
are $-1$, and both generalized Arf invariants are 0. Hence the assertion is
immediate in this case. Assume now that $u\leq2v(2)$, and write $c=1-h$ with $h
\in I_{0}$. By taking the images on both sides modulo $I_{u}$ and noting that
$c^{2}\gamma\delta\equiv\gamma\delta(\mathrm{mod\ }I_{u})$ we obtain
\[\gamma\delta \equiv h^{2}-2h+\alpha\beta(\mathrm{mod\ }I_{u}).\] We shall
consider the cases $v(h) \geq v(2)$ and $v(h)<v(2)$ separately. Recall that
$I_{u} \subseteq J_{u}$ for even $u$ and $I_{2v(2)}\subseteq4R_{AS}$, so that
the condition $\gamma\delta\equiv\alpha\beta(\mathrm{mod\ }I_{u})$ yields the
asserted conclusion for any (non-vanishing) generalized Arf invariant.

Now, as $u\leq2v(2)$, if $v(h)>v(2)$ then $h^{2}-2h \in I_{u}$ and
$\gamma\delta\equiv\alpha\beta(\mathrm{mod\ }I_{u})$. The same argument holds
if $v(h)=v(2)$ and $u<2v(2)$. The remaining case in which $v(h) \geq v(2)$ is
where $v(h)=v(2)$ and $u=2v(2)$. We write $h=2t$ with $t \in R$, and the
congruence shows that the difference between $\gamma\delta$ and $\alpha\beta$ is
$4\rho(t)$ modulo $I_{2v(2)}$. Proposition \ref{v=v4} shows that
$\alpha\beta+4\rho(t)$ is not in $I_{2v(2)}$ for any $t \in R$, so that the
corresponding class is non-zero, and Lemma \ref{RASFAS} completes the proof of
this case. Hence the assertion holds wherever $v(h) \geq v(2)$.

We now consider the case where $v(h)<v(2)$, which implies $v(2h)>2v(h)$. It
follows that if $2v(h)<u$ then the congruence cannot hold. This establishes the
inequality $2v(h) \geq u$, which in particular completes the proof for exact
generalized Arf invariants. We also have $2h \in I_{u}$, and we may omit it from
the congruence. Thus, if $2v(h)>u$ then we again have
$\gamma\delta\equiv\alpha\beta(\mathrm{mod\ }I_{u})$, which in particular
completes the proof for the case of odd $u$ since $2v(h) \geq u$ implies
$2v(h)>u$ in this case. It remains to consider the case where $u<2v(2)$ and is
even, and $2v(h)=u$. But our congruence shows that the difference
$\gamma\delta-\alpha\beta$ is $h^{2}$ modulo $I_{u}$, and as $2v(h)=u$ this
difference belongs to $J_{u}$. As Proposition \ref{v<v4} implies that the class
of $\alpha\beta+J_{u}$ is non-zero, this completes the proof of the proposition.
\end{proof}

The case $u=2v(2)$ in Proposition \ref{Arf} generalizes the Arf invariant
defined in \cite{[A]} for such lattices, whence the name. For more on this
relation, see Subsection \ref{Ex}. Note that this case (the exact generalized
Arf invariants) arises from \emph{non-zero} classes in $4R/4R_{AS}$; This is, in
some sense, complemented by the vanishing generalized Arf invariant,
representing the remaining, trivial class in $4R/4R_{AS}$. In any case, our
generalized Arf invariant carries also the additional information about the
valuation $v$. In Section \ref{Inv} we shall present a refinement of the
generalized Arf invariant in some cases, and use it to classify lattices
$M_{\alpha,\beta}$ satisfying some additional conditions.

We remark that we have defined the generalized Arf invariant only for rank 2
lattices, while the classical Arf invariant is defined for quadratic modules of
arbitrary even rank. As the generalized Arf invariant depends on some maximality
conditions, defining it for higher rank lattices requires much more care,
together with results of the same type as those appearing in Section
\ref{Ch2Can} below.

We remark that many results from this section remain valid when the 2-Henselian
assumption is relaxed. E.g., Lemma \ref{RASFAS} holds over $R$ also if we relax
the hypothesis $2\neq0$ in 2-Henselianity: The polynomial $f$ from the proof of
that lemma satisfies $f(x) \in I_{0}$ and $f'(x)=2x-1$ lies in $R^{*}$, so that
the usual Henselian property can be used to prove it. In this case, $\rho$ is a
homomorphism of additive groups also as a map on $R$. For Proposition
\ref{r2rep0} in the case where $2=0$ in $R$, we note that the only isotropic
vectors in $M_{\alpha,0}$ are multiples of $y$ (and there are no reflections
involved). Hence the set $(R/2R)/(R^{*})^{2}$ appearing in the classification
there becomes just the set $R/(R^{*})^{2}$, whose non-zero elements appeared in
the classification of rank 1 lattices. The remaining assertions do not use the
2-Henselian property, and we just remark that if $2=0$ in $R$ then a lattice
$M_{\alpha,\beta}$ with $\beta$ maximal is either isotropic or satisfies the
conditions of Proposition \ref{v<v4}. However, since 2-Henselianity is used in
Lemma \ref{iso}, and the rest of this Section uses the inequality
$v(\alpha)+v(\beta)\leq2v(2)$ for anisotropic lattices, we prefer to stay in the
2-Henselian setting.

\section{Invariants of Lattices with Primitive Norms in $2R$ \label{Inv}}

In this section we define a refinement of the generalized Arf invariant from
Definition \ref{GAIdef} in case the valuation is larger than $v(2)$. We then
present an additional invariant of lattices $M_{\alpha,\beta}$ containing
primitive elements with norms divisible by 2, and show that these two
invariants characterize isomorphism classes of such lattices. In the end of this
Section we reproduce the results for lattices over the 2-adic ring
$\mathbb{Z}_{2}$, and give also the example of lattices over
$\mathbb{Z}_{2}[\sqrt{2}]$.

\subsection{A Criterion for Isomorphism and Fine Arf Invariants
\label{FinArfInv}}

We first present the main criterion for isomorphism between lattices of the form
$M_{\alpha,\beta}$ with $v(\beta) \geq v(2)$.

\begin{lem}
Assume $v(\alpha) \leq v(2) \leq v(\beta) \leq v\big(\frac{4}{\alpha}\big)$. A
lattice $M_{\gamma,\delta}$, where $\delta$ is a norm of primitive element with
maximal valuation, is isomorphic to $M_{\alpha,\beta}$ if and only if
$\gamma=c^{2}(\alpha+2r)$ for some $r \in R$ (with the restriction
$r\in\frac{2}{\beta}R_{AS}$ if $v(\beta)=v(2)$) and $c \in R^{*}$, and
$\delta=\frac{\beta+2a-\frac{\alpha+2r}{1-\alpha\beta}a^{2}}{c^{2}(1+2\beta r)}$
for some $a \in R$ with $2v(a) \geq v\big(\frac{\beta}{\alpha}\big)$. In
particular, if $\gamma$ is as above and $\delta-\frac{\beta}{c^{2}} \in
I_{2v(2)-v(\alpha)}$ then $M_{\gamma,\delta} \cong M_{\alpha,\beta}$. In the
limit case $v\big(\delta-\frac{\beta}{c^{2}}\big)=v\big(\frac{4}{\alpha}\big)$,
we obtain $M_{\gamma,\delta} \cong M_{\alpha,\beta}$ if and only if the element
$\frac{\gamma}{4c^{2}}(c^{2}\delta-\beta)$ of $R^{*}$ lies in $R_{AS}$. If
$v(\alpha)<v(2)$ this is equivalent to $\frac{\alpha}{4}(\beta-c^{2}\delta) \in
R_{AS}$. \label{r2unimodcls}
\end{lem}

\begin{proof}
Since we assume $v(\alpha) \leq v(\beta)$ and $v(\gamma) \leq v(\delta)$, we
find that $\gamma$ is the norm of an element $z \in M_{\alpha,\beta}$ of the
form $z=c(x+sy)$ with $c \in R^{*}$ and $s \in R$. Indeed, if
$v(\alpha)<v(\beta)$ this is clear, and if $v(\alpha)=v(2)=v(\beta)$ then the
assertion follows by using the reflection with respect to $x$ (the one
mentioned after Lemma \ref{norm2}) if necessary. Lemma \ref{norm2} now implies
that $\gamma$ has the form $c^{2}(\alpha+2r)$ with $c \in R^{*}$ and
$r=s+\frac{\beta}{2}s^{2} \in R$, and $r\in\frac{2}{\beta}R_{AS}$ if
$v(\beta)=v(2)$. The second basis element $w$ of $M_{\gamma,\delta}$ satisfies
$(z,w)=1$, which implies
\[c(1+s\beta)w=y+\frac{a}{1-\alpha\beta}\big((1+s\beta)x-(\alpha+s)y\big)\] for
some $a \in R$. This is because $(y,z)=c(1+s\beta)$ and
$(1+s\beta)x-(\alpha+s)y$ spans the space $(Rz)^{\perp}$. Therefore
$\delta=\frac{[c(1+s\beta)w]^{2}}{c^{2}(1+s\beta)^{2}}=\frac{\beta+2a-\frac{
\alpha+2r}{1-\alpha\beta}a^{2}}{c^{2}(1+2\beta r)}$ as asserted, and the
inequality $2v(a) \geq v\big(\frac{\beta}{\alpha}\big)$ follows from the
maximality of $\delta$ (and implies $v(\delta)=v(\beta)$---see Propositions
\ref{v<v4} and \ref{v=v4}). Conversely, the map taking the two basis elements
of $M_{\gamma,\delta}$ with such $\gamma$ and $\delta$ to $z=c(x+sy)$ and the
asserted $w$ defines an isomorphism to $M_{\alpha,\beta}$ (the surjectivity of
this map follows either by evaluating the determinant of this change of basis or
just from unimodularity). This establishes the first assertion. Now, given
$\gamma$ (hence $c^{2}$), it remains to show that if $\delta-\frac{\beta}{c^{2}}
\in I_{2v(2)-v(\alpha)}$, or if
$v\big(\delta-\frac{\beta}{c^{2}}\big)=v\big(\frac{4}{\alpha}\big)$ and
$\frac{\gamma}{4c^{2}}(c^{2}\delta-\beta)$ lies in $R_{AS}$, then we can find an
appropriate value of $a$. Write $\frac{\alpha+2r}{1-\alpha\beta}$ as
$\eta\alpha$ for some $\eta$ (which belongs to $R^{*}$ since
$\eta=\frac{\gamma}{c^{2}\alpha(1-\alpha\beta)}$ and $v(\gamma)=v(\alpha)$), and
we need to find $a \in R$ such that $c^{2}(1+2\beta r)\delta=\beta+2a-\eta\alpha
a^{2}$. Denoting the left hand side by $\lambda$, we find
$v\big(\eta-\frac{\gamma}{c^{2}\alpha}\big)=v(\alpha\beta)>0$ and
$v(\lambda-c^{2}\delta) \geq v(2\beta\delta)>v\big(\frac{4}{\alpha}\big)$.
Next, we write $a=\frac{\lambda-\beta}{2}b$, and look for a solution for
\[0=-\eta\alpha
a^{2}+2a+(\beta-\lambda)=(\beta-\lambda)\big[\eta\alpha\frac{(\lambda-\beta)}{4}
b^{2}-b+1\big].\] If $\delta-\frac{\beta}{c^{2}} \in I_{2v(2)-v(\alpha)}$ then
$A=\eta\alpha\frac{(\lambda-\beta)}{4} \in I_{0}$, and Lemma \ref{quadeq} with
$B=-1$ and $C=1$ yields a solution $b \in R$. If
$v\big(\delta-\frac{\beta}{c^{2}}\big)=v\big(\frac{4}{\alpha}\big)$ then
$\frac{4}{\alpha}(\lambda-\beta) \in R^{*}$, so that we write
$b=\frac{4}{\eta\alpha(\lambda-\beta)}h$ and the equation becomes
$0=\frac{4}{\eta\alpha(\lambda-\beta)}\big[h^{2}-h+\eta\alpha\frac{
(\lambda-\beta)}{4}\big]$. The approximations for $\lambda$ and $\eta$ above
show that $\eta\alpha\frac{(\lambda-\beta)}{4}$ has the same $\mathbb{F}$-image
as the element $\frac{\gamma}{4c^{2}}(c^{2}\delta-\beta)$ of $R_{AS}$, so that
Lemma \ref{RASFAS} implies the existence of a solution to the latter equation.
Finally, if $v(\alpha)<v(2)$ then
$\eta$, hence also $\frac{\gamma}{c^{2}\alpha}$, are in $1+I_{0}$, so that
$\frac{\gamma}{4c^{2}}(\beta-c^{2}\delta) \in R_{AS}$ if and only if
$\frac{\alpha}{4}(c^{2}\delta-\beta) \in R_{AS}$. This proves the lemma.
\end{proof}

Lemma \ref{r2unimodcls} is the main tool for investigating whether two lattices
of the form $M_{\alpha,\beta}$ and $M_{\gamma,\delta}$, with $\beta$ and
$\delta$ maximal with valuations at least $v(2)$, are isomorphic. We begin by
showing that isomorphism classes of lattices with generalized Arf invariant in
$I_{v(2)}$ can be described using yet another invariant, which is finer than the
generalized Arf invariant.

\begin{prop}
$(i)$ The set $S=\{t^{2}-2t|t \in R,2v(t)>v(2)\}$ is a subgroup of $I_{v(2)}$
which contains $I_{2v(2)}$. $(ii)$ If $M_{\alpha,\beta} \cong M_{\gamma,\delta}$
with $\beta$ and $\delta$ maximal and the (common) generalized Arf invariant
is contained in $I_{v(2)}$, then $\alpha\beta$ and $\gamma\delta$ coincide
modulo $S$. $(iii)$ There exists a well-defined map from elements of the
quotient $I_{v(2)}/S$ containing a representative of maximal valuation onto the
set of generalized Arf invariants with valuation larger than $v(2)$.
\label{finArf}
\end{prop}

\begin{proof}
$(i)$ The inclusion $S \subseteq I_{v(2)}$ is clear. For $y \in I_{2v(2)}$ we
have $\frac{y}{4} \in I_{0} \subseteq R_{AS}$, and if $\frac{y}{4}=\rho(x)$
then $y$ is obtained as $t^{2}-2t$ for $t=2x$. In order to show that $S$ is a
subgroup, we show that the difference between two elements $t^{2}-2t$ and
$s^{2}-2s$ of $S$ also lies in $S$, since it is of the form
$(t-s+h)^{2}-2(t-s+h)$ for some $h$ such that $2v(h)>v(2)$. Indeed, comparing
terms yields the equation $h^{2}+2(t-s-1)h+2s^{2}-2ts=0$ for $h$, where the
coefficients $A=1$, $B=2(t-s-1)$, and $C=2s^{2}-2ts$ satisfy $v(A)=0$,
$v(B)=v(2)$, and $v(C)>2v(2)$. To see the inequality concerning $v(C)$, observe
that $v(s^{2})>v(2)$ and $v(ts)>v(2)$ because $2v(t)+2v(s)>2v(2)$ by our
assumptions on $s$ and $t$. Lemma \ref{quadeq} thus yields a solution $h$ of
valuation $v\big(\frac{C}{B}\big)$, which is the same valuation as
$v(ts-s^{2})>v(2)$. Hence $2v(t-s+h)>v(2)$ and the difference is indeed an
element of $S$.

$(ii)$ Let now $M_{\alpha,\beta}$ and $M_{\gamma,\delta}$ be isomorphic lattices
such that the common valuation $u$ of $\alpha\beta$ and $\gamma\delta$ satisfies
$u>v(2)$. Then $1-\gamma\delta=c^{2}(1-\alpha\beta)$ for some $c \in R^{*}$, and
an argument similar to the proof of Proposition \ref{Arf} shows that $c$ must be
of the form $1-h$ with $2v(h) \geq u$. But then $\alpha\beta(h^{2}-2h) \in
I_{2v(2)} \subseteq S$ and $h^{2}-2h \in S$, so that $\gamma\delta-\alpha\beta
\in S$ as asserted.

$(iii)$ Let $\omega$ be a representative of a class of $I_{v(2)}/S$ of maximal
valuation, i.e., $v(\omega+t^{2}-2t) \leq v(\omega)$ for every $t$ with
$2v(t)>v(2)$. If $v(\omega)>2v(2)$ then $\omega \in S$ hence $\omega=0$.
Otherwise, the set of such representatives of the class $\omega+S$ is just
$\omega+(S\cap\omega R)$, and it contains no element of $I_{v(\omega)}$.
Considerations similar to those presented in the proof of Proposition \ref{Arf}
yield the following conclusions: If $v(\omega)<2v(2)$ and
is odd then $S\cap\omega R \subseteq I_{v(\omega)}$. If $v(\omega)<2v(2)$ but is
even, then the image of $S\cap\omega R$ modulo $I_{v(\omega)}$ coincides with
that of $J_{v(\omega)}$. Finally, if $v(t^{2}-2t)\geq2v(2)$ then $t \in 2R$, and
if $t=2r$ then $t^{2}-2t=4\rho(r)$. The latter observation implies the equality
$S\cap4R=4R_{AS}$ (and in particular we see that $I_{2v(2)} \subseteq S$ again).
The maximality of $v(\omega)$ in its class implies that $\omega$ must satisfy
the conditions for $\alpha\beta$ in Propositions \ref{v<v4} and \ref{v=v4}
(i.e., either $v(\omega)<2v(2)$ is odd, or $v(\omega)<2v(2)$ is even and
$\omega$ is not an approximate square, or $v(\omega)=2v(2)$ and $\omega$ not
lying in $4R_{AS}$). In particular every such $\omega$ defines a generalized Arf
invariant. Moreover, our arguments show that this generalized Arf invariant is
the image of $\omega$ in the quotient of $\omega$ modulo the group $(S\cap\omega
R)+I_{v(\omega)}$, which is coarser than the quotient modulo $(S\cap\omega R)$
and coincides with it if $v(\omega)=2v(2)$. Hence the map taking a class in
$I_{v(2)}/S$ containing an element $\omega$ of maximal valuation to the
generalized Arf invariant represented by $\omega$ is well-defined, as it just
take the image of $\omega$ in one quotient to the image of $\omega$ in a coarser
quotient. This proves the proposition.
\end{proof}

On the basis of these arguments, we make the following
\begin{defn}
A \emph{fine Arf invariant} is defined to be the set of representatives of
maximal valuation in a class in $I_{v(2)}/S$ containing such representatives.
The \emph{valuation of a fine Arf invariant} is the valuation of any such
representative. A fine Arf invariant is said to be
\begin{enumerate}
\item \emph{vanishing} if it comes from the trivial class with the
representative 0;
\item \emph{odd} if its valuation is smaller than $2v(2)$ and odd;
\item \emph{even} in case its valuation is smaller than $2v(2)$ and even; and
\item \emph{exact} in case its valuation equals precisely $2v(2)$.
\end{enumerate}
Given a lattice $M_{\alpha,\beta}$ with $\beta$ of maximal valuation as the norm
of a primitive element in this lattice and with generalized Arf invariant of
valuation larger than $v(2)$, we define the \emph{fine Arf invariant of
$M_{\alpha,\beta}$} to be the set of elements of valuation $v(\alpha\beta)$ in
the class $\alpha\beta$ modulo $S$. \label{FAIdef}
\end{defn}
Part $(ii)$ of Proposition \ref{finArf} shows that the fine Arf invariant of a
lattice $M_{\alpha,\beta}$ whose generalized Arf invariant has valuation larger
than $v(2)$ is an invariant of the isomorphism class of $M_{\alpha,\beta}$. The
type (vanishing, odd, even, or exact) of a fine Arf invariant from Definition
\ref{FAIdef} coincides with the type of the generalized Arf invariant to which
is it taken by the map from part $(iii)$ of Proposition \ref{finArf}. This map
also preserves the valuations of fine and generalized Arf invariants, and its
restriction to vanishing and exact fine Arf invariants (i.e., to valuations at
least $2v(2)$) gives a natural bijection between these fine and generalized Arf
invariants. It is clear that this map takes the fine Arf invariant of a lattice
$M_{\alpha,\beta}$ (for which the fine Arf invariant is defined) to the
generalized Arf invariant of this lattice.

\subsection{Classes of Minimal Norms \label{MinNorm}}

Lemma \ref{r2unimodcls} shows that if $v(\beta) \geq v(2)$ and $M_{\alpha,\beta}
\cong M_{\gamma,\delta}$ then $\gamma$ differs from some element in
$(R^{*})^{2}\alpha$ by an element in $2R$ (and even in $\frac{4}{\beta}R_{AS}$
if $v(\beta)=v(2)$). Lemma \ref{RASFAS} implies that the set
$\frac{4}{\beta}R_{AS}$ depends only on the image of $\beta$ modulo
$I_{v(\beta)}$, so that we can write it (at least heuristically at this point)
as $\frac{4\alpha}{\eta}R_{AS}$ using the generalized Arf invariant $\eta$ of
$M_{\alpha,\beta}$. Thus, if $v(\beta)=v(2)$ then $\gamma$ and $\alpha$ can be
described as related through the action of the multiplicative group
$(R^{*})^{2}\big(1+\frac{4}{\eta}R_{AS}\big)$. In order to put the relation for
$v(\beta)>v(2)$ on the same basis, we remark that (at least for the non-zero
classes with $v(\alpha)<v(2)$) the relation $\gamma\in(R^{*})^{2}(\alpha+2R)$
can also be phrased using the action of the group
$(R^{*})^{2}\big(1+\frac{2}{\xi}R\big)$ where $\xi \in R$ is any element of $R$
with $v(\xi)=v(\alpha)$. We therefore introduce the following
\begin{defn}
A \emph{coarse class of minimal norms} is an element of the set
$(R/2R)/(R^{*})^{2}$. Given a lattice $M_{\alpha,\beta}$ such that $v(\beta)$ is
maximal and satisfies $v(\beta)>v(2)$, we define the \emph{class of minimal
norms of $M_{\alpha,\beta}$} to be the image of $\alpha$ in the set of coarse
classes of minimal norms. Given a generalized Arf invariant $\eta$ with $v(\eta)
\leq 2v(2)$ (i.e., non-vanishing) we define a \emph{fine class of minimal norms
arising from $\eta$} to be the orbit of an element of $R$, of valuation $u$
satisfying $2u \leq v(\eta)$, under the action of the multiplicative group
$(R^{*})^{2}\big(1+\frac{4}{\tau}R_{AS}\big)$, where $\tau$ is some element of
$R$ whose class in the appropriate quotient is $\eta$. Let $M_{\alpha,\beta}$ be
a lattice with generalized Arf invariant $\eta$, and assume that $v(\beta)$ is
maximal and equals $v(2)$ (so that $v(\eta) \geq v(2)$). In this case we define
the \emph{class of minimal norms of $M_{\alpha,\beta}$} to be the image of
$\alpha$ in the set of fine classes of minimal norms arising from $\eta$.
\label{CMNdef}
\end{defn}
Note that the condition on $v(\beta)$ implies that if the generalized (or fine)
Arf invariant $\eta$ of $M_{\alpha,\beta}$ is not of vanishing type then the
class of minimal norms of $M_{\alpha,\beta}$ has valuation strictly smaller than
$v(\eta)-v(2)$ if $v(\beta)>v(2)$, and its valuation equals precisely
$v(\eta)-v(2)$ in case $v(\beta)=v(2)$.

The fact that the objects appearing in Definition \ref{CMNdef} are well-defined,
and the main classification result of this paper, are as follows:
\begin{thm}
The set of fine classes of minimal norms arising from each generalized Arf
invariant $\eta$ with $v(\eta) \leq 2v(2)$ is well-defined. Given a generalized
Arf invariant $\eta$ with $v(\eta)>v(2)$, the isomorphism classes of lattices
$M_{\alpha,\beta}$ with generalized Arf invariant $\eta$ and with primitive
elements with norm in $2R$ (of maximal valuation) are characterized precisely by
their fine Arf invariant and their classes of minimal norms. \label{v>v2cls}
\end{thm}

\begin{proof}
The set in question is defined as orbits under the action of the group
$(R^{*})^{2}\big(1+\frac{4}{\tau}R_{AS}\big)$, where $\tau$ represents the
generalized Arf invariant $\eta$. We must thus show that taking another
representative for $\eta$ yields the same group. Now, $\eta$ can be either odd,
even, or exact, so that we have to consider the effect of adding to $\tau$ an
element from $I_{v(\eta)}$, $J_{v(\eta)}$, or $4R_{AS}$ respectively.

Now, adding an element from $I_{v(\eta)}$ to $\tau$ is the same as dividing it
by something from $1+I_{0}$. The effect on $\frac{4}{\tau}R_{AS}$ is
multiplication of $R_{AS}$ by $1+I_{0}$, which leaves it invariant by Lemma
\ref{RASFAS}. In particular, the assertion for odd $\eta$ follows. For even
$\eta$ it remains to consider the effect of adding $\lambda^{2}$ to $\tau$,
where $2v(\lambda)=v(\eta)$. The group $1+\frac{4}{\tau+\lambda^{2}}R_{AS}$
will in general be different from $1+\frac{4}{\tau}R_{AS}$, but we claim that
it is contained in $(R^{*})^{2}\big(1+\frac{4}{\tau}R_{AS}\big)$. Indeed,
Proposition \ref{v<v4} shows that $\frac{\lambda^{2}}{\tau}\not\in1+I_{0}$, so
that $1+\frac{\lambda^{2}}{\tau} \in R^{*}$ and every element of
$1+\frac{4}{\tau+\lambda^{2}}R_{AS}$ may be written as
$1+\frac{4}{\tau+\lambda^{2}}\rho\big[\big(1+\frac{\lambda^{2}}{\tau}\big)r\big]
$ for some $r \in R$. But we may expand
\[\rho\bigg[\bigg(1+\frac{\lambda^{2}}{\tau}\bigg)r\bigg]=\bigg(1+\frac{\lambda^
{2}}{\tau}\bigg)\rho(r)+\frac{\lambda^{2}}{\tau}\bigg(1+\frac{\lambda^{2}}{\tau}
\bigg)r^{2}\] (a direct calculcation using the definition of $\rho$), so that
our element becomes
$1+\frac{4}{\tau}\rho(r)+\frac{4\lambda^{2}r^{2}}{\tau^{2}}$. As
$2v(\lambda)=v(\tau)=v(\eta)<2v(2)$, we find that $\frac{2\lambda r}{\tau} \in
I_{0}$ and $v\big(\frac{4\lambda
r}{\tau}\big)>v\big(\frac{4\lambda^{2}r^{2}}{\tau^{2}}\big)$, so that dividing
this element by $\big(1+\frac{2\lambda r}{\tau}\big)^{2}$ yields an element of
$1+\frac{4}{\tau}\rho(r)+I_{2v(2)-v(\tau)}$. But Lemma \ref{RASFAS} implies
that the latter expression takes the form $1+\frac{4}{\tau}\rho(s)$ for some $s
\in R$, so that our original expression equals
$\big(1+\frac{2\lambda r}{\tau}\big)^{2}\big(1+\frac{4}{\tau}\rho(s)\big)$ and
lies in $(R^{*})^{2}\big(1+\frac{4}{\tau}R_{AS}\big)$ as desired. Interchanging
the roles of $\tau$ and $\tau+\lambda^{2}$ now show that the groups arising
from both numbers coincide, which proves that sets of fine classes of minimal
norms arising from even $\eta$ are also well-defined.

When the generalized Arf invariant $\eta$ is exact, we need to verify that
adding an element from $4R_{AS}$ to $\tau$ leaves the group
$(R^{*})^{2}\big(1+\frac{4}{\tau}R_{AS}\big)$ invariant. Hence we consider an
element of the group $1+\frac{4}{\tau+4\rho(h)}R_{AS}$ for some $h \in R$. Note
that we may always replace $h$ by $1-h$, since they have the same $\rho$-image.
The number $\frac{4\rho(h)}{\tau}$ cannot be in $1+I_{0}$ by Proposition
\ref{v=v4}, so that $1+\frac{4\rho(h)}{\tau} \in R^{*}$ and we write an element
of our group as
$1+\frac{4}{\tau+4\rho(h)}\rho\big[\big(1+\frac{4\rho(h)}{\tau}\big)r\big]$.
Using a similar expansion, the latter expression equals
$1+\frac{4}{\tau}\rho(r)+\frac{16\rho(h)r^{2}}{\tau^{2}}$. Writing the two
$\rho$-images explicitly, we get
\[1+\frac{4}{\tau}(r^{2}-r)+\frac{16r^{2}}{\tau^{2}}(h^{2}-h)=1+\frac{16r^{2}h^{
2}}{\tau^{2}}+\frac{4r^{2}}{\tau}-\frac{4r}{\tau}\bigg(1+\frac{4rh}{\tau}
\bigg).\] We may assume, by replacing $h$ by $1-h$ if necessary, that
$1+\frac{4rh}{\tau} \in R^{*}$. When we divide this expression by
$\big(1+\frac{4rh}{\tau}\big)^{2}$, the sum of the first two terms give an
element of $1+I_{0}$ (which equals $1+\frac{4}{\tau}I_{0}$ since
$\frac{4}{\tau} \in R^{*}$), and the remaining three terms become just
$\frac{4}{\tau}\rho\big(\frac{r}{1+4rh/\tau}\big)$. Invoking Lemma \ref{RASFAS}
again yields an element $s \in R$ such that the whole sum is just
$1+\frac{4}{\tau}\rho(s)$, so that our original element equals
$\big(1+\frac{4rh}{\tau}\big)^{2}\big(1+\frac{4}{\tau}\rho(s)\big)$ and lies in
$(R^{*})^{2}\big(1+\frac{4}{\tau}R_{AS}\big)$. The symmetry between $\tau$ and
$\tau+4\rho(h)$ now establishes the invariance of the group
$(R^{*})^{2}\big(1+\frac{4}{\tau}R_{AS}\big)$ under this operation, so that the
sets of fine classes of minimal norms are well-defined also for exact $\eta$.
This proves the first assertion.

Now, Lemma \ref{r2unimodcls} shows that isomorphic lattices of the form
$M_{\alpha,\beta}$ with $v(\beta) \geq v(2)$ and $v(\alpha\beta)>v(2)$ have the
same class of minimal norms, and that no finer invariant for the minimal norm
exists. Moreover, part $(ii)$ of Proposition \ref{finArf} shows that the fine
Arf invariant is also an invariant of isomorphism classes of such lattices.
Conversely, assume that $M_{\alpha,\beta}$ and $M_{\gamma,\delta}$ have
the same fine Arf invariant and the same class of minimal norms. In particular,
the difference between the valuation of the (common) generalized Arf invariant
$\eta$ of these two lattices and $v(\alpha)$ coincides with its difference from
$v(\gamma)$, so that we consider either the coarse classes of minimal norms in
both lattices or the fine classes of minimal norms arising from $\eta$. Since
the (appropriate) classes of $\alpha$ and $\gamma$ coincide, Lemma
\ref{r2unimodcls} shows that $M_{\gamma,\delta}$ is isomorphic to
$M_{\alpha,\mu}$ for some $\mu \in R$. Moreover, by part $(ii)$ of Proposition
\ref{finArf} the lattice $M_{\alpha,\mu}$ has the same fine Arf invariant as
$M_{\gamma,\delta}$ and $M_{\alpha,\beta}$, meaning that $\alpha\beta-\alpha\mu
\in S$. If the fine (or generalized) Arf invariant vanishes then $\mu=\beta=0$
and we are done. Otherwise $v(\eta)\leq2v(2)$, and we write the difference
$\alpha\beta-\alpha\mu$ as $t^{2}-2t$. We claim that $2v(t) \geq v(\eta)$.
Indeed, otherwise $v(t)<v(2)$ and all the elements $\alpha\mu$, $\alpha\beta$,
and $2t$ have valuations larger than $v(t^{2})$, so that the equality
$\alpha\beta-\alpha\mu=t^{2}-2t$ cannot hold. Now, since
$v(\eta)=v(\alpha\beta)\geq2v(\alpha)$ we can write $t=\alpha b$ for $b \in R$
with $2v(b) \geq v\big(\frac{\beta}{\alpha}\big)$ and obtain the equality
$\mu=\beta+2b-\alpha b^{2}$. We claim that this equality implies an equality of
the form $\mu=\beta+2a-\frac{\alpha a^{2}}{1-\alpha\beta}$ for some $a \in R$
with $2v(a) \geq v\big(\frac{\beta}{\alpha}\big)$. Indeed, write $a=b+h$ in the
desired equality, and using the given relation between $\mu$, $\beta$, and $b$
we obtain the equation $Ah^{2}+Bh+C=0$ with $A=-\frac{\alpha}{1-\alpha\beta}$
(of valuation $v(\alpha)$), $B=2\big(1-\frac{\alpha b}{1-\alpha\beta}\big)$ (of
valuation $v(2)$---recall that $t=\alpha b$ and $\alpha\beta$ are both in
$I_{0}$), and $C=-\frac{\alpha^{2}b^{2}\beta}{1-\alpha\beta}$ (of valuation at
least $v(\alpha)+2v(\beta)$ by the condition on $b$). The inequalities
$v(\alpha\beta)=v(\eta)>v(2)$ and $v(\beta) \geq v(2)$ allow us to apply Lemma
\ref{quadeq}, and show that the valuation $v\big(\frac{C}{B}\big)$ of the
solution $h$ is larger than $v(2)$. This proves the existence of an appropriate
$a$, and Lemma \ref{r2unimodcls} completes the proof of the theorem.
\end{proof}

Since the proof of part $(iii)$ of Proposition \ref{finArf} shows that over an
exact or vanishing generalized Arf invariant there exists only one fine Arf
invariant, the result of Theorem \ref{v>v2cls} in these cases can be phrased as
in the following

\begin{cor}
$(i)$ Assume that $v(\beta)>v(2)$ and $v(\alpha)+v(\beta)\geq2v(2)$. Then the
class of $\alpha$ in $(R/2R)/(R^{*})^{2}$ and the generalized Arf invariant
$\eta$ (the image of $\alpha\beta$ modulo $4R_{AS}$) characterize the the
isomorphism class of $M_{\alpha,\beta}$, where the vanishing of the former
invariant implies the vanishing of the latter. $(ii)$ If
$v(\alpha)=v(\beta)=v(2)$ then the isomorphism classes of such lattices are
characterized by the generalized Arf invariant $\eta$, and in case $\eta\neq0$
(i.e., $\eta$ is exact), also by the class of elements of
$2R/(R^{*})^{2}\big(1+\frac{4}{\tau}R_{AS}\big)$ (of valuation 2), where
$\tau\in4R$ is such that $\eta=\tau+4R_{AS}$, containing all the norms of
primitive elements in a lattice in this isomorphism class. \label{clsv>v4}
\end{cor}

\begin{proof}
Part $(i)$ follows directly from Theorem \ref{v>v2cls}. Part $(ii)$ is obtained
by combining Theorem \ref{v>v2cls} and Lemma \ref{norm2}.
\end{proof}

The case of a vanishing generalized Arf invariant (namely, isotropic lattices)
of Theorem \ref{v>v2cls} and Corollary \ref{clsv>v4} reproduces the result of
Proposition \ref{r2rep0}, though the latter holds over any valuation ring while
Theorem \ref{v>v2cls} and Corollary \ref{clsv>v4} require the $2$-Henselian
property.

\smallskip

Once again, the lattices with classes of minimal norms of valuation 0 are
decomposable. Such a lattice $M_{\alpha,\beta}$, with $\tau \in R$ representing
the fine Arf invariant of $M_{\alpha,\beta}$, is isomorphic to the lattice
$H_{\alpha,\tau}$ spanned by two orthogonal elements $u$ and $w$ of norms
$\alpha$ and $\alpha(\tau-1)$: Indeed, by taking $\tau=\alpha\beta$ we find that
$x=u$ and $y=\frac{u+w}{\alpha}$ form a basis for $H_{\alpha,\tau}$ as an
isomorphic image of $M_{\alpha,\beta}$. Therefore Theorem \ref{v>v2cls} implies
also the following

\begin{cor}
Let $\alpha$ and $\gamma$ be elements in $R^{*}$ and let $\tau$ and $\lambda$
be elements in $I_{v(2)}$ representing generalized Arf invariants. The lattices
$H_{\alpha,\tau}$ and $H_{\gamma,\lambda}$ are isomorphic if and only if
$\tau-\lambda \in S$ and $\alpha$ and $\gamma$ are in the same coarse class of
minimal norms in $(R/2R)/(R^{*})^{2}$. \label{hypv>v2}
\end{cor}

As above, the case $\tau=\lambda=0$ in Corollary \ref{hypv>v2} yields Corollary
\ref{hypiso} (under the 2-Henselianity assumption).

\smallskip

Note that Theorem \ref{v>v2cls} and Corollary \ref{hypv>v2} deal only with
lattices whose generalized Arf invariants come from $I_{v(2)}$. Next we show
that these results extend to generalized Arf invariants with valuation
precisely $v(2)$ under some conditions on the 2-Henselian valuation ring $R$.
\begin{prop}
Assume that $R$ satisfies either $(i)$ $v(2)$ is odd, $(ii)$ $\mathbb{F}$ is
perfect, or $(iii)$ $\rho$ is surjective. Then all the lattices
$M_{\alpha,\beta}$ with $\beta\in2R$ having maximal valuation are classified by
their fine Arf invariant and the appropriate class of minimal norms.
\label{v=v2ext}
\end{prop}

\begin{proof}
Define $S^{+}=\{t^{2}-2t|t \in R,2v(t) \geq v(2)\}$. If we can show that
$S^{+}$ is a subgroup of $2R$, then Definition \ref{FAIdef} may be extended to
introduce fine Arf invariants with valuation precisely $v(2)$, using classes in
$2R/S^{+}$. Apart from this extension, there are two additional places in the
proofs of Proposition \ref{finArf} and Theorem \ref{v>v2cls} where we have used
strict inequalities for our arguments. One is where for a generalized Arf
invariant coming from $\alpha\beta \in I_{v(2)}$ and an element $h$ with
$2v(h)>v(2)$, the expression $\alpha\beta(h^{2}-2h)$ lies in $I_{2v(2)}
\subseteq S$ (part $(ii)$ of Proposition \ref{finArf}). The second place
appears at the end of the proof of Theorem \ref{v>v2cls}, where we concluded
that if $\alpha$ and $\beta$ are as above and $b$ satisfies $2v(b) \geq
v\big(\frac{\beta}{\alpha}\big)$, then the sum of the valuations of
$A=-\frac{\alpha}{1-\alpha\beta}$ and of
$C=-\frac{\alpha^{2}b^{2}\beta}{1-\alpha\beta}$ is larger than $2v(2)$, so that
a solution to the equation for $a$ exists by Lemma \ref{quadeq}. Note that the
latter sum $v(AC)$ is just $v(\alpha^{3}b^{2}\beta)$.

Now, in case $(i)$ the conditions $2v(t) \geq v(2)$ and $2v(h) \geq v(2)$
become strict inequalities since they compare the odd valuation $v(2)$ with an
even valuation. Hence $S^{+}=S$ is a group, and classes from $2R/S$ define fine
Arf invariants which are invariants of isomorphism classes of lattices (and map
to generalized Arf invariants as before). As for $v(\alpha^{3}b^{2}\beta)$, it
is at least $v(\alpha^{2}\beta^{2})$ since $2v(b) \geq
v\big(\frac{\beta}{\alpha}\big)$, hence $AC\in4R$. But if $v(\alpha\beta)>v(2)$
then we already have $AC \in I_{2v(2)}$, while in the case of equality
$v\big(\frac{\beta}{\alpha}\big)$ is odd, the equality with $v(b)$ is strict,
and again $AC \in I_{2v(2)}$. This proves for case $(i)$. Hence we consider
cases $(ii)$ and $(iii)$ under the additional assumption that $v(2)$ is even.

Case $(ii)$ with even $v(2)$ is simple: As there are no generalized Arf
invariants of even valuation $v(2)$ by Corollary \ref{v<v4perf}, there is no
need for any extension of the definitions, and all inequalities involving
$v(\alpha\beta)$ remain strict. In case $(iii)$ we consider again the equation
$h^{2}+2(t-s-1)h+2s^{2}-2ts=0$ for $h$ in the proof of part $(i)$ of Proposition
\ref{finArf}, where now $s$ and $t$ give rise to elements from $S^{+}$. By
writing $h=2(1-t+s)g$ this equation becomes
$g^{2}-g+\frac{s^{2}-st}{2(1-t+s)^{2}}=0$ (recall that $s^{2}$ and $st$ are in
$2R$ as above), and this Artin--Schreier equation has a solution by our
assumption on $\rho$ and Lemma \ref{RASFAS}. Moreover, as $S \cap 4R=4R_{AS}$
is the full ideal $4R$ in this case, the conditions $v(\alpha\beta) \geq v(2)$
and $2v(h) \geq v(2)$ are sufficient for $\alpha\beta(h^{2}-2h)$ to be in $4R
\subseteq S$. Hence $S^{+}$ is a group, fine Arf invariants of valuation $v(2)$
are well-defined, and they are preserved under isomorphism of lattices.
Finally, the equation for $a$ can be transformed by similar means to an
Artin--Schreier equation looking for a pre-image of
$\frac{\alpha^{3}b^{2}\beta}{4(1-\alpha\beta-\alpha b)}$, which again exists
under our assumption. This completes the proof of the proposition.
\end{proof}
We remark that extending Corollary \ref{hypv>v2} to the cases $(i)$ or $(iii)$
in Proposition \ref{v=v2ext} requires using the fine classes of minimal norms
arising from the common generalized Arf invariant arising from $\tau$ and
$\lambda$ in case $v(\lambda)=v(\tau)=v(2)$, rather than the coarse classes of
minimal norms appearing in that Corollary. The fact that there are no even
generalized Arf invariants in case $(ii)$ of Proposition \ref{v=v2ext} agrees,
for the case of $R$ is a (complete) discrete valuation ring, with the parity
condition on the weight and norm ideals in Section 93 of \cite{[O]}.

On the other hand, if none of the conditions of Proposition \ref{v=v2ext} is
satisfied, then $S^{+}$ is no longer a group, but generalized Arf invariants of
valuation $v(2)$ exist. Therefore, an appropriate definition of fine Arf
invariants of valuation $v(2)$, which will be preserved under isomorphisms,
requires much more care and will probably be more involved. Attempts to extend
to generalized Arf invariants of valuation smaller than $v(2)$ encounter more
severe difficulties (as no group structure is expected there), and will be
left for future research. However, we point out one fact that arises from the
proof of Lemma \ref{r2unimodcls} in this more general case: If $\delta-\beta$
lies in $I_{2v(2)-v(\alpha)}$ then $M_{\alpha,\beta} \cong M_{\alpha,\delta}$.
Indeed, putting $s=0$ and $c=1$ there (i.e., $z=x$) shows that the vector $w$
has norm $\beta+2a-\frac{\alpha a^{2}}{1-\alpha\beta}$ (note that the
maximality of $\beta$ shows that $1-\alpha\beta$ cannot be in $I_{0}$, by
Proposition \ref{v<v4}). As comparing this value to $\delta$ yields a quadratic
equation in which the sums of the valuations of
$A=-\frac{\alpha}{1-\alpha\beta}$ and $C=\beta-\delta$ is larger than twice the
valuation of $B=2$, Lemma \ref{quadeq} yields a solution $a$ to this equation,
which proves the assertion. Moreover, the assumption that $v(\beta) \geq
v(\alpha)$ was not used in this argument, so that we also deduce
$M_{\alpha,\beta} \cong M_{\gamma,\beta}$ if $\gamma-\alpha \in
I_{2v(2)-v(\beta)}$. This fact will be useful in completing the classification
for lattices over $\mathbb{Z}_{2}[\sqrt{2}]$ in Subsection \ref{Ex} below.

\subsection{Relations to Quadratic Forms and Examples \label{Ex}}

A notion closely related to (symmetric) bilinear forms, which has not appeared
in this paper yet, is the notion of quadratic forms. Recall that a
\emph{quadratic form} on an $R$-module $M$ is a map $q:M \to R$ which satisfies
$q(rx)=r^{2}x^{2}$ for all $r \in R$ and $x \in M$, and such that the map
taking $x$ and $y$ in $M$ to $q(x+y)-q(x)-q(y)$ is a bilinear form on $M$ (this
is the \emph{bilinear form coming from $q$}). We denote this bilinear form
$\varphi(q)$, so that we have a map $\varphi=\varphi_{M}$ from the set of
quadratic forms on $M$ to the set of bilinear forms on $M$. In case $2 \in
R^{*}$, every bilinear form comes from a unique quadratic form, namely
$q(x)=\frac{x^{2}}{2}$. Hence $\varphi$ is a canonical bijection. If $2 \notin
R^{*}$ but is not a zero-divisor in $R$ (i.e., $2\neq0$ and $R$ is an integral
domain), then $\varphi$ is injective, but may not be surjective. This is so,
since we can localize by 2 (making $\varphi$ bijective again), but some
bilinear forms which are $R$-valued will require the quadratic form to take
values in the localization. Those lattices in which the bilinear form comes from
a quadratic form via $\varphi$ are the lattices called \emph{even} in the
terminology of \cite{[Z]} and others. But other lattices exist: E.g.,
$M_{\alpha,\beta}$ is even precisely when $\alpha$ and $\beta$ are both in $2R$,
i.e., the generalized (or equivalently fine) Arf invariant is exact or
vanishing and the class of minimal norms comes from $2R$. In case $2=0$,
however, this map $\varphi$ is in general neither injective nor surjective (this
is the map considered in \cite{[A]} for $R$ a field of characteristic 2). Hence
a quadratic form on a module over an integral domain yields more information
than the one obtained using the bilinear form arising as its $\varphi$-image
alone only if the integral domain in question has a fraction field of
characteristic 2.

The last assertion of Corollary \ref{clsv>v4} shows the relation of exact and
vanishing generalized Arf invariants to the classical Arf invariants: Any field
$\mathbb{F}$ of characteristic 2 is the quotient field of a 2-Henselian
valuation ring $R$ whose fraction field $\mathbb{K}$ has characteristic 0 (for
an example of a such a ring in which the valuation is complete and discrete,
take $R$ to be the ring $W(\mathbb{F})$ of Witt vectors over
$\mathbb{F}$---see, e.g., Section II.6 of \cite{[Se2]} for more details on this
construction). Any non-degenerate (or \emph{fully regular} in the terminology of
\cite{[A]}) quadratic form of dimension 2 over $\mathbb{F}$ is isomorphic to a
form $q:(r,s)\mapsto\lambda r^{2}+rs+\mu s^{2}$ for some $\lambda$ and $\mu$ in
$\mathbb{F}$ (see Theorem 2 of \cite{[A]}---note that by normalizing one of the
basis elements we can make the product 1, i.e., we can take $b_{i}=1$ for all
$i$). Let $\alpha$ and $\beta$ be elements of $2R$ such that
$\frac{\alpha}{2}+I_{0}=\lambda$ and $\frac{\beta}{2}+I_{0}=\mu$. Then the
quadratic form $q$ can be seen as the reduction modulo $I_{0}$ of the map
$z\mapsto\frac{z^{2}}{2}$ for $z=rx+by \in M_{\alpha,\beta}$. If
$M_{\alpha,\beta} \cong M_{\gamma,\delta}$ for some $\gamma$ and $\delta$ in
$2R$ then $q$ is isomorphic over $\mathbb{F}$ to the quadratic form
$Q:(r,s)\mapsto\varphi r^{2}+rs+\psi s^{2}$ for $\varphi=\frac{\gamma}{2}+I_{0}$
and $\psi=\frac{\delta}{2}+I_{0}$ (by reducing the isomorphism modulo $I_{0}$).
On the other hand, Corollary \ref{clsv>v4} implies that the isomorphism class of
$M_{\alpha,\beta}$ is independent of the choice of $\alpha\in2\lambda+I_{v(2)}$
and $\beta\in2\mu+I_{v(2)}$. This implies that if $q$ and $Q$ above are
isomorphic over $\mathbb{F}$ then $M_{\alpha,\beta} \cong M_{\gamma,\delta}$:
Indeed, lifting the isomorphism over $\mathbb{F}$ to any map over $R$ yields an
isomorphism between $M_{\alpha,\beta}$ and $M_{\kappa,\nu}$ for
$\kappa\in2\varphi+I_{v(2)}$ and $\nu\in2\psi+I_{v(2)}$, and the previous
assertion implies $M_{\kappa,\nu} \cong M_{\gamma,\delta}$. Thus, isomorphism
classes of fully regular quadratic forms of rank 2 over $\mathbb{F}$ correspond
to isomorphism classes of lattices $M_{\alpha,\beta}$ over $R$, where $\alpha$
and $\beta$ are in $2R$ (i.e., of even lattices $M_{\alpha,\beta}$). Now, exact
or vanishing generalized Arf invariants are ``4 times'' the Arf invariant
$\Delta$ defined in \cite{[A]} (this means $4\Delta\in4R/4R_{AS}$ for
$\Delta\in\mathbb{F}/\mathbb{F}_{AS}$), and the set of numbers
$\frac{z^{2}}{2}+I_{0}$ obtained from primitive $z \in M_{\alpha,\beta}$ is
precisely the set of squares of non-zero elements of the odd part of the
Clifford algebra of $q$ over $\mathbb{F}$. Since Lemma \ref{norm2} implies that
this set is either $\mathbb{F}$ or an orbit in
$\mathbb{F}/(\mathbb{F}^{*})^{2}\big(1+\frac{4}{\tau}\mathbb{F}_{AS}\big)$
(after division by 2 and dividing modulo $I_{0}$), and the remainder of the
structure of the Clifford algebra is determined by the condition that the two
basis elements $x$ and $y$ are chosen such that $(x,y)=1$, Corollary
\ref{clsv>v4} implies Theorem 3 of \cite{[A]} (In fact, this normalization
shows that applying $\rho$ to the element $xy$ of the Clifford algebra yields
the Arf invariant $\Delta$, the condition about $\Delta$ in that Theorem is
redundant).

The results of Section \ref{Uni2} as well as this Section thus generalize the
classical assertions from \cite{[A]} to many lattices in which the bilinear
form does not necessarily come from a quadratic form. It seems likely that
a similar argument can treat binary quadratic forms over some valuation rings
in which $2=0$ and which are not fields---note that the results of this Section
are contained in Proposition \ref{r2rep0} in case $2=0$ since we assume here
$\beta\in2R$ throughout. We leave this question for future research.

\medskip

We now use our results in order to classify the unimodular rank 2 lattices over
two rings. Recall that every such lattice is isomorphic to some lattice
$M_{\alpha,\beta}$ with $v(\beta) \geq v(\alpha)$, and this lattice is
decomposable if and only if $\alpha$ is invertible, a case in which the lattice
is isomorphic to some lattice $H_{\alpha,\tau}$ (see the paragraph preceding
Corollary \ref{hypv>v2}). We start with $R=\mathbb{Z}_{2}$, the ring of 2-adic
integers. As $\rho(\mathbb{F}_{2})=0$, we have $R_{AS}=I_{0}=2R$. All the
lattices of the form $M_{\alpha,\beta}$ admit primitive vectors with norms of
maximal valuations, by Proposition \ref{maxv} or the finiteness of valuations
smaller than $v(2)$. There are three possible generalized Arf invariants: 0
(vanishing), $4+8R$ (exact), and $2+4R$ (odd). Moreover, the set
$(R/2R)/(R^{*})^{2}$ of coarse classes of minimal norms consists of two
elements, represented by 0 and 1, the latter having valuation 0. The condition
$2v(t)>v(2)$ in the definition of $S$ implies $2|t$ hence $S=8R$. Moreover, both
conditions $(i)$ and $(ii)$ of Proposition \ref{v=v2ext} are satisfied, so that
we can classify all the $\mathbb{Z}_{2}$-lattices $M_{\alpha,\beta}$ using our
method. The isotropic lattices are $M_{0,0}$ (the hyperbolic plane) and
$M_{1,0}$ (which is isomorphic to $H_{1,0}$, hence generated by 2 orthogonal
vectors having opposite norms). Over the exact generalized Arf invariant lies
the fine Arf invariant $4+8R$. A lattice having this fine Arf invariant with
$v(\beta)>v(2)$ must have 1 in its coarse class of minimal norms. Such a lattice
must therefore be isomorphic to $M_{1,4}$, which can be generated by orthogonal
elements of norms 1 and 3 as its isomorph $H_{1,4}$. The multiplicative group
corresponding to this generalized Arf invariant is just
$1+\frac{4}{4}R_{AS}=1+2R=\mathbb{Z}_{2}^{*}$. There is thus only one fine class
of minimal norms of valuation 1 arising from this generalized Arf invariant,
yielding the lattice $M_{2,2}$. The remaining fine Arf invariants are $2+8R$ and
$-2+8R$, lying over the odd generalized Arf invariant $2+4R$. Both have
valuation $v(2)=1$, so that we have to consider the fine classes of minimal
norms of valuation 0 corresponding to $2+4R$. The group by which we divide is
$1+\frac{4}{2}R_{AS}=1+4R$, so that there are two such classes, represented by 1
and $-1$. The invariants 1 and $2+8R$ yield a lattice isomorphic to $M_{1,2}$,
which is generated by two orthogonal elements of norm 1 like $H_{1,2}$. With the
invariants $-1$ and $2+8R$ comes the lattice $M_{-1,-2}$, an orthogonal basis of
which can be taken with both norms $-1$ (consider $H_{-1,-2}$). Taking now the
class with 1 and fine Arf invariant $-2+8R$ yields the lattice $M_{1,-2}$, a
basis of its isomorph $H_{1,-2}$ has norms 1 and $-3$. The last lattice, with
invariants $-1$ and $-2+8R$, must be isomorphic to $M_{-1,2}$, which, being
isomorphic to $H_{-1,2}$, has an orthogonal basis with elements of norms $-1$
and 3. One can easily verify that these results reproduce the results of
\cite{[J2]} for rank 2 unimodular 2-adic lattices, since $-3\equiv5(\mathrm{mod\
}8)$, $-1\equiv7(\mathrm{mod\ }8)$, and the lattice $M_{-1,-2}$ can be written
as $M_{3,-2}$ (since $-1\equiv3(\mathrm{mod\ }4)$) and the isomorphic lattice
$H_{3,-2}$ has an orthonormal basis consisting of two norm 3 vectors.

We now turn to present the explicit picture our results yield for the ring
$R=\mathbb{Z}_{2}[\sqrt{2}]$. The fine Arf invariants we obtain have valuation
larger than $v(2)$ (Corollary \ref{v<v4perf} or condition $(ii)$ of Proposition
\ref{v=v2ext}), and once again $R_{AS}=I_{0}$, which here equals $\sqrt{2}R$.
Proposition \ref{maxv} (or the fact that only finitely many positive elements of
$\Gamma$ are smaller than $v(2)=2$) shows again that in every lattice
$M_{\alpha,\beta}$ we can take $\beta$ to have maximal valuation. The elements
in $1+4\sqrt{2}R$ are squares, and $(R^{*})^{2}/(1+4\sqrt{2}R)$ consists of one
additional non-trivial class, which is represented by
$(1+\sqrt{2})^{2}=3+2\sqrt{2}$. A generalized Arf invariant of a lattice can be
0 (vanishing), $4+4\sqrt{2}R$ (exact), $2\sqrt{2}+4R$ (odd), or $\sqrt{2}+2R$
(odd). Fine Arf invariants are defined only above the first three generalized
Arf invariants. The group $S$ is based on elements satisfying $2v(t)>v(2)$,
which means $2|t$ hence $S=4\sqrt{2}R$. As the action of
$(R^{*})^{2}\subseteq1+2R$ on $R/2R$ is trivial, there are 4 coarse classes of
minimal norms, represented by 0, $\sqrt{2}$, 1, and $1+\sqrt{2}$, with
valuations $\infty$, 1, 0, and 0 respectively. Hence there are 4 isomorphism
classes of isotropic $R$-lattices, in which the generalized and fine Arf
invariants are vanishing, namely $M_{0,0}$, $M_{\sqrt{2},0}$, $M_{1,0}$, and
$M_{1+\sqrt{2},0}$. Considering the exact generalized and fine Arf invariant
$4+4\sqrt{2}R$, as all the non-zero coarse classes of minimal norms have
valuation smaller than $v(2)=2$, we obtain 3 isomorphism classes of non-even
lattices having this fine Arf invariant, which are represented by
$M_{\sqrt{2},2\sqrt{2}}$, $M_{1,4}$, and $M_{1+\sqrt{2},4}$ (recall that
$\frac{4}{1+\sqrt{2}}=4\sqrt{2}-4$ is congruent to 4 modulo $4\sqrt{2}R$). When
considering fine classes of minimal norms arising from this generalized Arf
invariant, the acting group is just $1+\sqrt{2}R=R^{*}$, so that there is only
one element of valuation 2 in this set, giving rise to the lattice $M_{2,2}$.
The generalized Arf invariant $2\sqrt{2}+4R$ appears as the image of two fine
Arf invariants, one being $2\sqrt{2}+4\sqrt{2}R$, and the other one is
$2\sqrt{2}+4+4\sqrt{2}R$. The isomorphism classes of lattices
$M_{\alpha,\beta}$ with $v(\beta)\geq3$ having these fine Arf invariants are
represented by the coarse classes of minimal norms 1 and $1+\sqrt{2}$. The
corresponding lattices are (up to isomorphism) $M_{1,2\sqrt{2}}$ and
$M_{1+\sqrt{2},2\sqrt{2}+4}$ with the fine Arf invariant
$2\sqrt{2}+4\sqrt{2}R$, while $M_{1,2\sqrt{2}+4}$ and
$M_{1+\sqrt{2},2\sqrt{2}}$ have the fine Arf invariant $2\sqrt{2}+4+4\sqrt{2}R$.
The fine classes of minimal norms arising from the generalized Arf invariant
$2\sqrt{2}+4R$ are obtained modulo the action of the group $1+2R$ (which
already contains $(R^{*})^{2}$). As the classes of valuation 1 are represented
by $\sqrt{2}$ and $2+\sqrt{2}$, we obtain the two additional lattices
$M_{\sqrt{2},2}$ and $M_{2+\sqrt{2},2+2\sqrt{2}}$ with the fine Arf invariant
$2\sqrt{2}+4\sqrt{2}R$, together with the lattices $M_{\sqrt{2},2+2\sqrt{2}}$
and $M_{2+\sqrt{2},2}$ having the fine Arf invariant $2\sqrt{2}+4+4\sqrt{2}R$.
Theorem \ref{v>v2cls} shows that these 16 isomorphism classes of $R$-lattices
are distinct, and every unimodular $R$-lattice admitting a primitive vector of
norm in $2R$ belongs to one of these isomorphism classes.

We end this section by completing the classification of those unimodular rank 2
lattices over $R=\mathbb{Z}_{2}[\sqrt{2}]$ to which Theorem \ref{v>v2cls} does
not apply. These lattices all have generalized Arf invariant $\sqrt{2}+2R$, and
they are all decomposable. By the remark at the end of Subsection \ref{MinNorm}
they take the form $M_{\alpha,\beta}$ where $\alpha$ can be taken from any set
of representatives for $R^{*}/(1+4R)$ (there are 8 such classes) and for $\beta$
one may use any set of representatives for the classes in $\sqrt{2}R/4\sqrt{2}R$
having valuation 1 (again 8 such classes). We identify, for the moment, these
sets of representatives with the corresponding classes, so that the two
operations we introduce below on these classes may be considered to be
normalized to always take representatives to representatives. There are 64
pairs in $R^{*}/(1+4R)\times\sqrt{2}R^{*}/(1+4R)$, and there are three
operations on these pairs such that two pairs are connected through these
operations if and only if they yield isomorphic lattices. The first operation
takes $\alpha$ and $\beta$ to $r^{2}\alpha$ and $\frac{\beta}{r^{2}}$ for $r \in
R^{*}$. This operation has exponent 2 here since $(R^{*})^{2}/(1+4R)$ has order
2. In addition, we may replace $\alpha$ by $\alpha+2s+\beta s^{2}$ for $s \in
R$, and $\beta$ by $\frac{\beta}{(1+\beta s)^{2}}$. The valuations of $\alpha$
and $\beta$ and the fact that we consider elements modulo multiplication from
$1+4R$ shows that this operation depends only on the class of $s$ in
$\mathbb{F}=\mathbb{F}_{2}$, yielding another operation of order 2. In
addition, we can take $\beta$ to $\beta+2t+\alpha t^{2}$ and $\alpha$ to
$\frac{\alpha}{(1+\alpha t)^{2}}$ for $t \in R$ with $v(t)>0$. By letting
$\gamma^{2}$ represent the non-trivial class in $(R^{*})^{2}/(1+4R)$ and
choosing $s=-\alpha$ and $t=-\beta$ to represent the non-trivial choices of the
two latter operations, we find that our operations, which we denote $\zeta$,
$\sigma$, and $\tau$, send the pair $(\alpha,\beta)$ to
$(\alpha\gamma^{2},\beta\gamma^{2})$,
$\big(\alpha(\alpha\beta-1),\beta\gamma^{2}\big)$, and
$\big(\alpha\gamma^{2},\beta(\alpha\beta-1)\big)$ respectively. Working modulo
$1+4R$ one sees that $\zeta$ is central, all three have order 2, and the
commutator of $\sigma$ and $\tau$ is $\zeta$. Hence these three operations
generate a dihedral group of order 8, which operates on our set of 64 pairs
without fixed points. It follows that there are 8 orbits, and representatives
for these orbits can be taken to be $M_{1,\sqrt{2}}$, $M_{1,\sqrt{2}+4}$,
$M_{1,-\sqrt{2}}$, $M_{1,-\sqrt{2}+4}$, $M_{-1,\sqrt{2}}$, $M_{-1,\sqrt{2}+4}$,
$M_{-1,-\sqrt{2}}$, and $M_{-1,-\sqrt{2}+4}$ (i.e., and $R$-lattice with
generalized Arf invariant $\sqrt{2}+2R$ is isomorphic to precisely one of these
8 lattices). One can find ad-hoc invariants: The closest analogue of the fine
Arf invariant would be elements in $\sqrt{2}+2\sqrt{2}R$ modulo $4\sqrt{2}R$,
but not for every lattice $M_{\alpha,\beta}$ the product $\alpha\beta$ lies in
this set (sometimes one has to switch to an isomorphic lattice), and the choice
of the elements $\sqrt{2}+2\sqrt{2}R$ rather than $\sqrt{2}+2+2\sqrt{2}R$ does
not have an immediate extension to the general case. As for classes of minimal
norms (which depend on $\alpha\beta$), 1 and $1+\alpha\beta$ (and their images
after multiplying by squares) lie in one class while $-1$ and $-1+\alpha\beta$
(times squares) lie in another class, and this situation does not seem to have a
clear description of the sort of Definition \ref{CMNdef}. This illustrates how
the existence of norms in $2R$ simplifies our method substantially.

\section{Towards Canonical Forms in Residue Characteristic 2 \label{Ch2Can}}

In this section we derive some relations between lattices of 2-Henselian
valuation rings in which $v(2)>0$. The idea is to give canonical
representatives for isomorphism classes of such lattices. This goal remains far
out of reach, but we give some results toward it.

The Jordan decomposition of a lattice $M$, given in Proposition \ref{decom}, is,
in this case, not unique. However, the different Jordan decompositions yielding
isomorphic lattices do have some properties in common:

\begin{prop}
Let $M=\bigoplus_{k=1}^{t}M_{k}$ and $M=\bigoplus_{k=1}^{t}N_{k}$ be two Jordan
decompositons of the same lattice $M$, with $v(M_{k})=v(N_{k})$ for every $k$
(allowing empty components if necessary). Then the uni-valued lattices $N_{k}$
and $M_{k}$ have the same rank (in particular, no empty components are needed in
two such presentations), and one of them has a diagonal basis if and only if
the other one has such a basis. \label{JorInv}
\end{prop}

\begin{proof}
We use the same method as in Section 93 of \cite{[O]}. Take some $0 \leq
v\in\Gamma$, and consider the subset $M_{v}$ of all elements $x \in M$ such
that $v(x,y) \geq v$ for every $y \in M$. We claim that this is a sub-lattice of
$M$. Note that since $R$ is not necessarily Noetherian (because the valuation is
not discrete), the assertion does not follow from the fact that $M^{v}$ is a
submodule of $M$: Indeed, the condition $v(x,y) \geq v$ can be interpreted as
$(x,y)$ being in the principal ideal $J$ of elements with valuation at least
$v$, and if we apply this condition for a non-principal ideal $J$ then the
resulting subset is not a finitely generated submodule of $M$. Now, if $M$
decomposes as $L \oplus N$ then $M_{v}$ decomposes as $L_{v} \oplus N_{v}$, so
that in particular $M_{v}$ decomposes either as $\bigoplus_{k=1}^{t}M_{k,v}$ or
as $\bigoplus_{k=1}^{t}N_{k,v}$. Let $a \in R$ with $v(a)=v$. If $N$ is
uni-valued, say $N=L(\sigma)$ with $L$ unimodular and $v(\sigma)=v(N)$, then
$N_{v}=N$ if $v(N) \geq v$ and $N_{v}=\frac{a}{\sigma}N$ if $v(N) \leq v$: The
first assertion is obvious, and the second assertion holds because any primitive
element $x \in N$ satisfies $\{(x,y)|y \in N\}=\sigma R$ since $N$ is
uni-valued. It follows that $M_{k,v}$ (or $N_{k,v}$) are lattices for every $k$,
and $M_{v}$ is a sub-lattice of $M$.

The lattice $M_{v}$ has valuation at least $v$. Moreover, its decompositions as
$\bigoplus_{k=1}^{t}M_{k,v}$ or $\bigoplus_{k=1}^{t}N_{k,v}$ are decompositions
to uni-valued lattices (but not necessarily in increasing valuation orders). To
see this, we examine $N_{v}$ for the uni-valued lattice $N=L(\sigma)$ again. If
$v(N) \geq v$ then $N_{v}=N=L(\sigma)$ is uni-valued with valuation
$v(N)=v(\sigma)$, while if $v(N) \leq v$ then $N_{v}=\frac{a}{\sigma}N$ is
isomorphic to $L\big(\frac{a^{2}}{\sigma}\big)$ and has valuation $2v-v(N)$. In
particular, $v(N_{v})>v$ unless $v(N)=v$. Consider now
$M_{v}\big(\frac{1}{a}\big)$. The inequality $v(M_{v}) \geq v=v(a)$ shows that
it is still a lattice, and we take the tensor product of this lattice with
$\mathbb{F}$. The images of all the components $M_{k,v}$ or $N_{k,v}$ with
$v(M_{k})=v(N_{k}) \neq v$ become degenerate in this $\mathbb{F}$-valued
bilinear form, and a maximal non-degenerate subspace of this $\mathbb{F}$-vector
space arises from $M_{k}$ or $N_{k}$ in case $v(M_{k})=v(N_{k})=v$. In
particular the ranks of $M_{k}$ and $N_{k}$ coincide for each $k$, and Corollary
\ref{diag} shows that each of them has an orthogonal basis if and only if some
element of $M_{v}\big(\frac{1}{a}\big)$ has a norm not in $I_{0}$. This proves
the proposition.
\end{proof}

A more detailed examination of the proof of Proposition \ref{JorInv} allows one
to derive a stronger assertion. Define, for each $2 \leq k<t$, the element
$u_{k}$ of $\Gamma$ to be $\min\{v(M_{k})-v(M_{k-1}),v(M_{k+1})-v(M_{k})\}>0$,
and for the extremal values $k=1$ and $k=t$ let $u_{1}=v(M_{2})-v(M_{1})$ and
$u_{t}=v(M_{t})-v(M_{t-1})$. Replacing $\mathbb{F}=R/I_{0}$ by $R/b_{k}R$ with
$b_{k} \in I_{0}$ having valuation $u_{k}$ shows that the images of
$M_{k}\big(\frac{1}{a}\big)$ and $N_{k}\big(\frac{1}{a}\big)$ modulo $b_{k}R$
are isomorphic. This implies

\begin{cor}
$(i)$ The two sets $\big\{x^{2}+ab_{k}R\big|x \in M_{k}\big\}$ and
$\big\{y^{2}+ab_{k}R\big|y \in N_{k}\big\}$ are the same subset of $R/ab_{k}R$.
$(ii)$ If $u_{k}>2v(2)$ then $M_{k} \cong N_{k}$. \label{minnormcomp}
\end{cor}

\begin{proof}
Part $(i)$ follows directly from the isomorphism $M_{k}\big(\frac{1}{a}\big)
\cong N_{k}\big(\frac{1}{a}\big)$ modulo $b_{k}R$ (alternatively, this set is
just the images of all norms from $M_{v}$ modulo $ab_{k}R$, and it is contained
in $aR/ab_{k}R$). Part $(ii)$ is a consequence of this isomorphism and Theorem
\ref{approxiso}.
\end{proof}

In fact, if $2 \in R^{*}$ then the condition $u_{k}>2v(2)$ is satisfied for any
$k$. Thus, Corollary \ref{minnormcomp} yields another proof of Theorem
\ref{isocomp}.

\smallskip

We shall define an order on the set of Jordan decompositions of lattices,
together with explicit forms of the components, in which one such Jordan
decomposition (with additional data) is larger than another one if it is more
canonical in the sense explained below. The idea is to define certain unimodular
components to be more canonical than others, and certain forms of such a
component as more canonical than other forms of the same component. After fixing
$\sigma_{v}$ for every $v$, the more canonical uni-valued components of
valuation $v$ are the more canonical unimodular ones with the bilinear form
multiplied by $\sigma_{v}$. For a general lattice $M$, we say that the Jordan
decomposition $\bigoplus_{k=1}^{t}M_{k}$ of $M$ is more canonical than
$\bigoplus_{k=1}^{t}N_{k}$ if there exists some $1 \leq l \leq t$ such that
$M_{k}=N_{k}$ for all $k<l$ and $M_{l}$ is more canonical than $N_{l}$. A
canonical form of a lattice $M$ would be a Jordan decomposition of $M$, with the
components given in a specific form, which is more canonical than any other
Jordan decomposition of $M$ or any other expression for the compoenents.

Before we give the details, we present the case $R=\mathbb{Z}_{2}$ considered in
\cite{[J2]}. In this case the unimodular components have symbols resembling
those arising from lattices over $\mathbb{Z}_{p}$ for odd $p$. The symbols take
the form $1^{\varepsilon n}_{t}$ or $1^{\varepsilon n}_{II}$, where $n$ is again
the rank and $\varepsilon$ is a Legendre symbol related to the discriminant of
the lattice. The additional index $t$ denotes an \emph{odd} (or \emph{properly
primitive} in the terminology of \cite{[J2]} and others) component, admitting an
orthonormal basis in which $t\in\mathbb{Z}/8\mathbb{Z}$ is the image of the sum
of the norms of the elements of such a basis. On the other hand, an index $II$
means that the component is \emph{even} (or \emph{improperly primitive}), i.e.,
admitting no orthonormal basis (see Corollary \ref{diag}). The parameter
$\lambda$ of \cite{[J2]} equals
$\varepsilon\cdot(-1)^{\frac{(n-t-4)(n-t-6)}{8}}$ in this
notation. Uni-valued lattices have symbols $(2^{k})^{\varepsilon
n}_{t}$ or $(2^{k})^{\varepsilon n}_{II}$ (standing for the unimodular lattice
$1^{\varepsilon n}_{t}$ or $1^{\varepsilon n}_{II}$ with the bilinear form
multiplied by $2^{k}$), and a Jordan decomposition of a general 2-adic lattice
is a product of such expressions (like for $p$-adic lattices for odd $p$),
yielding again a symbol for the lattice (with the chosen Jordan decomposition).
However, in this case different symbols can give rise to isomorphic lattices
(or equivalently, two different Jordan decompositions of the same lattice may
yield different symbols for the same lattice). Now, \cite{[J2]} considers a rank
1 lattice over $R=\mathbb{Z}_{2}$ representing $(R^{*})^{2}=1+8\mathbb{Z}_{2}$
to be more canonical than the other unimodular rank 1 lattices, and $M_{0,0}$
(whose generalized Arf invariant is 0 hence is vanishing) to be more canonical
than $M_{2,2}$ (with exact generalized Arf invariant $4+8\mathbb{Z}_{2}$).
Moreover, here $\Gamma=\mathbb{Z}$, and for $0 \leq v\in\Gamma$ we take
$\sigma_{v}=2^{v}$. \cite{[J2]} shows how to define the order of being more
canonical on all possible Jordan decompositions, and the canonical form of a
lattice $M$ appearing in Theorem 1 of \cite{[J2]} is the Jordan decomposition of
$M$ which is the most canonical one in this order. We remark that the order
depends on the (arbitrary) choice, which of $3+8\mathbb{Z}_{2}$ and
$7+8\mathbb{Z}_{2}$ is more canonical, a choice which is harder to generalize in
the arguments below (a choice of similar type appears also in the classification
of unimodular rank 2 lattices over $R=\mathbb{Z}_{2}[\sqrt{2}]$ having
generalized Arf invariant $\sqrt{2}+2R$ at the end of Subsection \ref{Ex}).

We now define when one form of a unimodular lattice is more canonical than
another form, for lattices over a 2-Henselian valuation ring $R$ in which
$v(2)>0$. First, a canonical form is based either on an orthogonal basis (if it
exists) or of a direct sum of lattices of the form $M_{\alpha,\beta}$ with
$\alpha$ and $\beta$ in $I_{0}$ (the proof of Corollary \ref{diag} shows that
such a form always exists for a unimodular lattice). In order to define the
further relations in the order, we say that an element $f \in R^{*}$ is
\emph{closer to 1} than $g \in R^{*}$ if $v(f-1)>v(g-1)$. Now, one form of a
lattice is more canonical than another form of the same lattice (or from a form
of a different lattice) if it has discriminant closer to 1. If the lattice
admits an orthongonal basis, then one orthongonal basis is more canonical than
another if it contains more elements of norms in $1+I_{0}$. If two bases have
the same number of elements with norms in  $1+I_{0}$, we order the base such
that the elements whose norms are closer to 1 come first. Then the basis
$x_{j}$, $1 \leq j \leq n$ is more canonical then $y_{j}$, $1 \leq j \leq n$ if
there exists some $1 \leq k \leq n$ such that $x_{k}^{2}$ is closer to 1 than
$y_{k}^{2}$ and $v(x_{j}^{2}-1)=v(y_{j}^{2}-1)$ for all $j<k$. In particular,
an orthonormal basis containing a maximal set of elements of norm precisely 1
will be more canonical than a basis not having this property. For a lattice of
the form $\bigoplus_{j}M_{\alpha_{j},\beta_{j}}$, we choose the order such that
the valuations of generalized Arf invariants are decreasing. Then one form is
more canonical than another using a condition similar to the orthogonal base
case, with ``$x^{2}$ being closer to 1'' replaced by ``the generalized Arf
invariant having higher valuation''. For two unimodular lattices given in a
certain form, we call one of them more canonical than the other according to the
same rules.

Let $L$ be a unimodular lattice over $R$ (a 2-Henselian valuation ring with
$v(2)>0$) having an orthonormal basis. Using the argument of Section
\ref{Jordan}, we find that the reduction of $L$ modulo $I_{0}$ decomposes as the
orthogonal direct sum of elements of norms $1+I_{0}$, and an
$\mathbb{F}$-lattice in which no norm equals $1+I_{0}$. Lifting this basis to a
basis of $L$ and altering by elements of $I_{0}$, we obtain an orthogonal basis
$x_{j}$, $1 \leq j \leq n$ for $M$ in which $x_{j}^{2}\in1+I_{0}$ for $j \leq
l$, and no combination of $x_{j}$ with $l+1<j \leq n$ has a norm in $1+I_{0}$.
If $\mathbb{F}$ is perfect than any non-zero norm is a square times an element
of $1+I_{0}$, so that the reduction can always be taken orthonormal and $l=n$.
If $M$ is a lattice with Jordan decomposition $M=\bigoplus_{k=1}^{t}M_{k}$ such
that $M_{1}=L$, then mixing with the components $M_{k}$ of higher valuation
cannot yield norms from $\bigoplus_{j=l+1}^{n}Rx_{j}\oplus\bigoplus_{k>1}M_{k}$
which are in $1+I_{0}$, hence cannot render our form of $L=M_{1}$ more
canonical.

We now turn to unimodular rank one components generated by an element with a
norm in $1+I_{0}$. Recall that a more canonical form for such a lattice will be
based on a generator $x$  whose norm is such that $v(x^{2}-1)$ is large. Now, if
we can have a most canonical form for such a lattice (i.e., $x^{2}=1+r$ with the
maximal possible $v(r)$, which is thus positive) then either $r=0$, $v(r)<2v(2)$
and is odd, $v(r)<2v(2)$ is even and $r$ is not in $\sigma^{2}R^{2}+I_{v(r)}$
for $\sigma \in R$ with $2v(\sigma)=v(r)$, or $v(r)=2v(2)$ and $r$ is not in
$4R_{AS}$. Indeed, if $v(r)>2v(2)$ then $x^{2}$ is a square by Lemma
\ref{quadeq}. Otherwise, we compare $1+r$ to $c^{2}(1+r)$, $c$ has to be $1+h$
for $h \in I_{0}$ with $2v(h) \geq v(r)$, and considerations like those
presented in Section \ref{Uni2} prove the assertion.

The first step towards a canonical form is provided by the following

\begin{prop}
Let $M$ be a unimodular lattice generated by two orthogonal elements $x$ and
$y$, whose norms are norms $1+r$ and $1+s$ respectively. Assume that $v(s) \geq
v(r)>0$, $v(s)$ is maximal, and $r$ is such that $v(r)$ is maximal among the
norms of primitive generators of $(Ry)^{\perp}$. If $v(r)$ is smaller than both
$v(s)$ and $v(2)$ then the generalized Arf invariant which $r$ represents is an
invariant of the lattice. If $v(r)+v(s)>2v(2)$ then $M$ is isomorphic to a
lattice spanned by two orthogonal elements of norms 1 and $(1+r)(1+s)$
respectively, hence $s=0$ by maximality. If $s\neq0$ but $v(s)$ is maximal (and
$v(s)>0$) then $r$ and $s$ satisfy the conditions of Propositions \ref{v<v4} or
\ref{v=v4} with $\alpha=r$ and $\beta=s$. \label{rep1}
\end{prop}

\begin{proof}
For any $t \in R$, the elements $x+(1+r)ty$ and $y-(1+s)tx$ are orthogonal and
have norms $(1+r)[1+t^{2}(1+r)(1+s)]$ and $(1+s)[1+t^{2}(1+r)(1+s)]$
respectively. We take only $t$ which is not in $1+I_{0}$, so that these vectors
are primitive and generate $M$. We can thus divide these elements by $1+t$, and
obtain generators of $M$ having norms $1+u$ and $1+w$ with
\[u=\frac{t^{2}(1+s)(r^{2}+2r)+r-2t+st^{2}}{(1+t)^{2}},\quad
w=\frac{t^{2}(1+r)(s^{2}+2s)+s-2t+rt^{2}}{(1+t)^{2}}.\] Every presentation of
$M$ with an orthogonal basis is obtained in this way, up to multiplying $1+u$
and $1+w$ by elements from $(R^{*})^{2}$: This follows directly from primitivity
and orthogonality. We are looking for isomorphic presentations of $M$ with $v(w)
\geq v(s)$. Hence if $v(r)<v(s)$ we take $t$ either with $2v(t) \geq
v\big(\frac{s}{r}\big)>0$ or $v(t) \geq v\big(\frac{2}{r}\big)>0$. It follows
that $u$ represents the same generalized Arf invariant as $r$, and by the usual
maximality argument this generalized Arf invariant is an invariant of the
isomorphism class of this lattice. If $v(r)+v(s)>2v(2)$ (hence $v(s)>v(2)$) then
the equation $w=0$ is quadratic in $t$, with $A$ of valuation at least $v(r)$
and with $B=-2$ and $C=s$. Lemma \ref{quadeq} gives a solution $t$ of valuation
$v\big(\frac{s}{2}\big)>0$ to this equation, so that $w=0$ can indeed be
obtained. The evaluation of $1+u$ as $(1+r)(1+s)$ is carried out either using
the equation for $t$ or using discriminant considerations. Now assume that
$s\neq0$ and $v(s)$ is maximal (so that no element of $M$ has norm precisely 1).
The maximality of $s$ implies that $v(s-2t+rt^{2}) \leq v(s)$ for every $t \in
R\setminus(1+I_{0})$, since the denominator in the expression for $w$ is in
$R^{*}$ and the other term in the numerator has valuation larger than $v(s)$.
Arguments similar to those of Section \ref{Uni2} now complete the proof of the
proposition.
\end{proof}

A slight modification of the proof of Proposition \ref{rep1} shows that if a
lattice $L$ admits an orthonormal basis with norms in $1+I_{0}$ such that one
norm is $1+r$ with $v(r)<v(2)$ and all the other norms are closer to 1 that
$1+r$ then $r$ defines a generalized Arf invariant which is an invariant of $L$.

The effect of combining a lattice $M_{\alpha,\beta}$ (with $\alpha$ and $\beta$
in $I_{0}$) with a unimodular rank 1 lattice spanned by a vector $z$ with
$z^{2}=u \in R^{*}$ is already considered in the proof of Corollary \ref{diag}.
The norms of the three basis elements given there in this case are $u+\alpha
t^{2}$, $t^{2}u+u^{2}\beta$, and $-(u+\alpha
t^{2})(t^{2}+u\beta)(1-\alpha\beta)$ respectively. Modulo $I_{0}$, these norms
are $u+I_{0}$, $t^{2}u+I_{0}$, and $-t^{2}u+I_{0}$, which are all equivalent to
$u+I_{0}$ modulo $(R^{*})^{2}$ since $\mathbb{F}$ has characteristic 2.

\smallskip

We now examine the effect of adding a lattice of positive valuation to a
unimodular lattice.

\begin{prop}
Let $M$ be a unimodular lattice whose discriminant in a given basis is in
$1+I_{0}$, and let $L$ be a lattice with $v(L)>0$. Write the discriminant of
$M$ in this basis as $1+r$, and assume that there exist primitive elements $x
\in M$ and $y \in L$, with norms $a$ and $b$ respectively, such that $t^{2}ab
\in r+I_{v(r)}$ for some $t \in R$. Then, the lattice $M \oplus L$ is isomorphic
to $N \oplus K$ with $N$ having the same reduction modulo $I_{0}$ as $M$ and has
discriminant $1+s$ with $v(s)>v(r)$, and $K$ is a lattice with the same
valuation as $L$ and whose discriminant is the discriminant of $L$ multiplied
by $1+t^{2}ab$. \label{umod+I0}
\end{prop}

\begin{proof}
Let $x_{i}$, $1 \leq i \leq n$ be a basis for $M$ giving the discriminant $1+r$
and in which $x_{n}^{2}=a$, and let $y_{j}$, $1 \leq j \leq m$ be a basis for
$L$ with $y_{m}^{2}=b$. Consider the elements $z_{i}=x_{i}+t(x_{i},x_{n})y_{m}$
and $w_{j}=y_{j}-t(y_{j},y_{m})x_{n}$ of $M \oplus L$. One verifies that $z_{i}
\perp w_{j}$ for all $i$ and $j$, and
$(z_{i},z_{j})\equiv(x_{i},x_{j})(\mathrm{mod\ }I_{0})$ since $y_{m}^{2}=b \in
I_{0}$ as $v(L)>0$. We take $N$ to be the lattice spanned by $z_{i}$, $1 \leq i
\leq n$ and $K$ to be the lattice generated by $w_{j}$, $1 \leq j \leq m$. If
$D$ is the matrix defined by $d_{ij}=(x_{i},x_{j})$ (with determinant $1+r$ by
assumption) then the matrix giving the discriminant of $N$ is
$D+t^{2}bd_{n}d_{n}^{t}$, where $d_{n}$ is the last column of $D$. By the Matrix
Determinant Lemma, the determinant of this matrix is $(1+t^{2}b \cdot
d_{n}^{t}D^{-1}d_{n})\det D$, and since $D^{-1}d_{n}$ is the $n$th standard
basis vector and $d_{nn}=x_{n}^{2}=a$, this expression reduces to
$(1+t^{2}ab)(1+r)$. Writing the latter expression as $1+s$ with
$s=r+t^{2}ab+rt^{2}ab$ and observing that $r$ and $t^{2}ab$ are in $I_{0}$,
$t^{2}ab \in r+I_{v(r)}$, and $\mathbb{F}$ has characteristic $2$, we find that
$v(s)>v(r)$. For elements of the lattice $K$ the expression $(w_{j},w_{k})$
differs from $(y_{j},y_{k})$ by $at(y_{j},y_{m})(y_{k},y_{m})$ of valuation at
least $2v(L)$. Evaluating the discriminant of $K$ can be carried out in the same
way as the discriminant of $N$ (since $L$ is non-degenerate, the corresponding
matrix has non-zero determinant hence can be inverted over $\mathbb{K}$). This
completes the proof of the proposition.
\end{proof}

Proposition \ref{umod+I0} can be used in various manners in order to convert a
component in a Jordan decomposition of a lattice into a more canonical one,
while affecting only the Jordan  components with higher valuations. As two
possible examples, consider the following: A lattice element with norm $1+r$
(contained in a unimodular component) with $v(r)$ maximal can be taken to an
element of norm $1+s$ with $v(s)>v(r)$, and a lattice $M_{\alpha,\beta}$ can be
altered in this way to a lattice $M_{\gamma,\delta}$ with generalized Arf
invariant of higher valuation. In some cases, the class of minimal norms can
also change to a class with larger valuation. All the transformations
presented in \cite{[J2]} over $R=\mathbb{Z}_{2}$ are special cases of
Propositions \ref{rep1} and \ref{umod+I0}.

\noindent\textsc{Fachbereich Mathematik, AG 5, Technische Universit\"{a}t
Darmstadt, Schlossgartenstrasse 7, D-64289, Darmstadt, Germany}

\noindent E-mail address: zemel@mathematik.tu-darmstadt.de

\end{document}